\documentclass[12pt,reqno]{amsart}

\usepackage{amscd} 
\usepackage{amsfonts} 
\usepackage{amsmath} 
\usepackage{amsrefs} 
\usepackage{amssymb} 
\usepackage{amsthm} 
\usepackage{latexsym} 
\usepackage{mathrsfs} 
\usepackage{mathtools} 
\mathtoolsset{showonlyrefs} 
\usepackage{slashed}  

\usepackage{verbatim}
\usepackage{cases}
\usepackage[dvips]{epsfig}
\usepackage{epsf} 
\usepackage[bookmarksnumbered,pdfpagelabels=true,plainpages=false,colorlinks=true,
            linkcolor=black,citecolor=black,urlcolor=black]{hyperref}
\usepackage{url}
\usepackage{color}
\usepackage{xcolor}
\usepackage{graphicx}
\usepackage{enumitem} 

\usepackage{pifont} 
\usepackage{upgreek} 
\usepackage{fancyhdr} 
\usepackage{calligra} 
\usepackage{marvosym} 
\usepackage[percent]{overpic} 
\usepackage{pict2e} 
\usepackage[hypcap=false]{caption} 
\usepackage{wasysym} 
\usepackage{accents}

\usepackage[papersize={8.5in,11in}, margin=1in,footskip=0.4in,vmarginratio={1:1},hmarginratio={1:1}]{geometry} 
\newcommand{\la}{\langle}
\newcommand{\ra}{\rangle}

\newcommand{\cl}{\mathcal}

\theoremstyle{plain}
\newtheorem{theorem}{Theorem}[section]

\newtheorem{proposition}[theorem]{Proposition}
\newtheorem{corollary}[theorem]{Corollary}

\theoremstyle{definition}
\newtheorem{remark}[theorem]{Remark}

\newtheorem{definition}[theorem]{Definition}

\numberwithin{equation}{section}
\allowdisplaybreaks

\newcommand{\norm}[1]{\|#1\|}

\newcommand{\mss}{\hspace{0.2cm}}
\newcommand{\ms}{\hspace{0.25cm}}

\newcommand{\cA}{\mathcal{A}}
\newcommand{\cB}{\mathcal{B}}
\newcommand{\cC}{\mathcal{C}}
\newcommand{\cF}{\mathcal{F}}
\newcommand{\cL}{\mathcal{L}}
\newcommand{\cR}{\mathcal{R}}
\newcommand{\cS}{\mathcal{S}}

\newcommand{\bC}{\mathbb{C}}
\newcommand{\bE}{\mathbb{E}}
\newcommand{\bR}{\mathbb{R}}

\newcommand{\fR}{\mathfrak{R}}

\newcommand{\Proj}{\mathsf{\Pi}}

\newcommand{\ii}{\mathrm{i}}

\begin{document}
\title[Barotropic viscous relativistic hydrodynamics]{Local well-posedness in Sobolev spaces for first-order
barotropic causal relativistic viscous hydrodynamics}
\author[Bemfica, Disconzi, Graber]{
Fabio S. Bemfica$^{* \$}$, 
Marcelo M. Disconzi$^{*** \#}$, 
and P.~Jameson Graber$^{**** \ddagger}$
}

\thanks{$^{\$}$FSB gratefully acknowledges support from a Discovery grant administered by Vanderbilt University. Part of this work was done while FSB was visiting Vanderbilt University.
}

\thanks{$^{\#}$MMD gratefully acknowledges support from a Sloan Research Fellowship 
provided by the Alfred P. Sloan foundation, from NSF grant DMS-1812826,
from a Discovery grant administered by Vanderbilt University, and from
a Dean's Faculty Fellowship.
}

\thanks{$^{\ddagger}$PJG gratefully acknowledges support from NSF grant DMS-1905449.
}

\thanks{$^{*}$Universidade Federal do Rio Grande do Norte, Natal, RN, Brazil.
\texttt{fabio.bemfica@ect.ufrn.br}}

\thanks{$^{***}$Vanderbilt University, Nashville, TN, USA.
\texttt{marcelo.disconzi@vanderbilt.edu}}

\thanks{$^{****}$Baylor University, Baylor, TX, USA.
\texttt{Jameson\_Graber@baylor.edu}}

\begin{abstract}
We study the theory of relativistic 
viscous hydrodynamics introduced in \cite{DisconziBemficaNoronhaBarotropic,Kovtun:2019hdm}, which
provided a causal and stable first-order theory of relativistic fluids with viscosity
in the case of barotropic fluids.
The local well-posedness of its equations of motion has been previously established
in Gevrey spaces. Here, we improve this result by proving local well-posedness
in Sobolev spaces.

\bigskip


\bigskip


\end{abstract}

\maketitle

\tableofcontents

\section{Introduction\label{S:Intro}}
Relativistic fluid dynamics is widely used in many branches of physics,
including high-energy nuclear physics \cite{Baier:2007ix}, astrophysics \cite{RezzollaZanottiBookRelHydro-2013}, and cosmology \cite{WeinbergBookCosmology-2008}.
Its power stems from conservation laws, such as the local conservation of energy and 
momentum, which allow one to investigate the macroscopic dynamics of conserved quantities without knowing the fate of the system's microscopic degrees of freedom.
In other words, although the complete behavior of physical systems is ultimately determined
by the dynamics of its microscopic constituents, one can bypass the usually
intractable problem of solving the full microscopic dynamics and work instead
within the scope of the the so-called \emph{fluid approximation.} The latter is understood
as a regime determined by energy scales where the system's microscopic 
constituents behave collectively as a continuum, which is then identified as a 
fluid \cite{DeGroot:1980dk}. While there remain questions about the details
of how to fully derive relativistic fluid dynamics from an underlying microscopic 
theory \cite{Bhattacharyya:2008jc,Heller:2016rtz,Romatschke:2017ejr,
Baier:2007ix,Denicol:2016bjh,Denicol:2012cn},
and rigorous mathematical results in this direction are few 
\cite{Speck-Strain-2011,Elskens-Kiessling-2020},
the overwhelming success of the relativistic fluid dynamics more than justifies the importance of studying
its mathematical properties. Furthermore, from a purely mathematical point of view, 
relativistic fluid dynamics has also been 
a fertile source of mathematical problems  (see, e.g.,
\cite{ChristodoulouBookShocks-2007,
ChristodoulouBookShocks-2019,
Choquet-BruhatBook-2009,
AnileBook-1990,
RezzollaZanottiBookRelHydro-2013,
DisconziSpeckRelEulerNull}
 and references therein). 

The first works on relativistic fluids go back to the early days of relativity theory with the works 
of Einstein \cite{Einstein:1914bx} and Schwarzchild \cite{Schwarzschild-1916}. The first
general\footnote{By ``general" we mean outside symmetry classes or beyond
one spatial dimension. With symmetry or in $1+1$ dimensions, the equations of relativistic fluid dynamics studied
by Choquet-Bruhat and Lichnerowicz reduce to equations for which earlier
techniques could have been applied, although it seems difficult to locate in the literature
specific applications of such known techniques to the equations of relativistic fluids
under symmetry assumptions or in one spatial dimension.} mathematical treatment of relativistic fluids was done by Choquet-Bruhat 
\cite{Choquet-Bruhat-1958} and Lichnerowicz \cite{LichnerowiczBookMHD-1967}.
Such works, as well as most of the studies in relativistic fluid dynamics since then, focused
on perfect fluids, i.e., fluids where viscosity and heat dissipation are absent\footnote{The literature
on this topic is quite large and an appropriate review is beyond the scope of this work. See
the literature cited in the first paragraph of this introduction and references therein
for further information.}. The equations describing relativistic perfect fluids are the well-known
relativistic Euler equations.

There are, however, important situations in physics where the relativistic Euler equations
are not appropriate, and a model of relativistic
fluids \emph{with viscosity} is needed. One such situation is in the study 
of the quark-gluon plasma, which is an exotic type of fluid forming in collisions
of heavy ions performed at the Large Hadron Collider (LHC) at CERN
and at the Relativistic Heavy Ion Collider (RHIC) at Brookhaven National Laboratory. 
Its discovery was named by the American Physical Society one of the 10 most important 
findings in physics in the last decade \cite{Chodos-APSNewsQGP-2010} and continues
to be a source of scientific breakthroughs \cite{STAR:2017ckg,PhysicsORG-NewsMostVorticalFluid}.
For the quark-gluon plasma, it is well-attested that theoretical 
predictions do not match
experimental data if viscosity is not taken into account \cite{Heinz:2013th,Romatschke:2017ejr}.
Another case, where viscosity is likely to play an important role, is in the study of gravitational waves
produced by neutron star mergers, 
which have been detected by
the Laser Interferometer Gravitational-Wave Observatory (LIGO)
\cite{LIGOVirgoPressReleaseNeutronStars,TheLIGOScientific:2017qsa,Monitor:2017mdv,GBM:2017lvd,
Abbott:2018exr}. Recent state-of-the-art numerical simulations
\cite{Alford:2017rxf,Shibata:2017xht,Shibata:2017jyf} convincingly show that 
the post-merger gravitational wave signal is likely to be affected by viscous effects.
Thus, one has two of the most cutting-edge 
experimental apparatuses in modern science (LHC and LIGO) producing data
that \emph{requires relativistic fluids with viscosity}
for its explanation\footnote{Although it is not claimed
that \emph{all} the data generated in these experiments can only be explained with viscosity.}.
But despite the importance of
relativistic viscous fluids, many essential questions remain unanswered and
very little is known about their mathematical properties.

Unlike the case of perfect fluids, it remains open what the best model 
for the description of relativistic viscous fluids is\footnote{It is interesting to notice 
that the generalization of the classical Navier-Stokes to a general Riemannian manifold is also
somewhat problematic, and there are different possible choices for the equations, see
\cite{DisconziChanCzubak}.}. 
This is because it is challenging to construct theories
of relativistic viscous fluids that are (i) causal, (ii) stable, and (iii) locally well-posed
\cite{RezzollaZanottiBookRelHydro-2013}. Causality
is a fundamental property of relativity stating that no information propagates faster 
than the speed of light. Stability here means linear stability about constant 
equilibrium states, i.e., mode stability, which on physical grounds
is expected to hold when viscous dissipation is present. Local well-posedness 
assures that the equations of motion admit a unique solution, a 
crucial property for physical models\footnote{Local well-posedness is also important
in the study of convergence of numerical schemes. The relation between local well-posedness
and convergence is subtle and a discussion of this topic is outside the scope
of this work. Interested readers can consult \cite{Guermond-Marpeau-Popov-2008} and references 
therein. This is an important topic since many studies of realistic physical systems
do rely on numerical computations.}. 
One requires (i) and (iii) to hold both in a fixed background and when
the fluid equations are coupled to Einstein's equations, whereas (ii) is usually required
only in Minkowski background (thus, all references to stability in what follows refer to the
equations in Minkowski space).

The first theory of relativistic viscous fluids was introduced by Eckart \cite{Eckart:1940te},
followed by a similar theory by Landau and Lifshitz \cite{LandauLifshitzBookFluids-1987}.
While these theories can be viewed as the simplest generalization of the classical Navier-Stokes
equations to the relativistic setting, they have been showed to be acausal and unstable 
\cite{Hiscock:1985zz,PichonViscousFluids-1965}. The M\"ueller-Israel-Stweart theory
originally introduced in the references \cite{Mueller-1967,Israel:1976tn,Israel:1979wp} is an attempt
to overcome the acausality and instability of the Eckart and Landau-Lifshitz theories and
is based on extended irreversible thermodynamics \cite{Jou-Vazquez-Lebon-Book-2010,Muller-Ruggeri-Book-1998}.
In this formalism, viscous  and dissipative contributions to the fluid's energy-momentum tensor
are not given in terms of standard \emph{hydrodynamic variables,} which are the 
fluid's velocity, energy density, baryon density, and quantities derived from these\footnote{Given
an equation of state, whose form depends on the nature of the fluid, all thermodynamic scalars (such as
energy density, entropy, temperature, pressure, etc.) are related via the laws of thermodynamics and
only two of them are independent. Absent phase
transitions, all such relations are invertible and the choice of which two thermodynamic scalars are 
independent is a matter of convenience.}. Rather, in extended thermodynamic theories, 
viscous and dissipative contributions are modeled by new variables, commonly referred to as
\emph{extended variables}, which satisfy further equations
of motion. In the original works of M\"uller, Israel, and Stewart, such equations of motion
were chosen in order to enforce the second-law of thermodynamics. In modern versions
of the M\"uller-Israel-Stewart theory\footnote{Strictly speaking, these modern derivations do not 
exactly reproduce the original M\"uller-Israel-Stewart equations, but are close enough so that
it has become common practice to still call them M\"uller-Israel-Stewart, although sometimes they 
are also referred by another names (re-summed \cite{Baier:2007ix} 
BRSSS or DNMR \cite{Ryu:2017qzn,Denicol:2012cn}). All such theories
are based on extended variables and behave very similarly
when it comes to issues of stability and causality, so that
it does not seem important to distinguish them here.}, the equations of motion are derived from
microscopic theory or are based on effective theory arguments (see below for further remarks on
the derivation of fluid equations from microscopit theory) \cite{Baier:2007ix,Ryu:2017qzn,Denicol:2012cn}.
Theories where viscous or dissipative effects are modeled solely by the hydrodynamics variables 
are known as \emph{first-order theories,} whereas those where such effects are modeled 
by extended variables are known as \emph{second-order theories} \cite{RezzollaZanottiBookRelHydro-2013}.

The M\"uller-Israel-Stewart theory has been proved to be stable and its linearization
about constant equilibrium states is causal \cite{Hiscock:1983zz,Olson:1989ey}
Furthermore, it has been extensive applied to the construction of successful phenomenological
models of the quark-gluon plasma \cite{Heinz:2013th,Romatschke:2017ejr}. 
Consequently, the M\"uller-Israel-Stewart theory is currently the most used 
theory for the description of relativistic viscous fluids. More recently, 
it has been proved that the M\"uller-Israel-Stewart theory with 
bulk viscosity\footnote{As
in the case of classical fluids, there are generally two types of viscosity in relativistic fluids, namely,
shear viscosity and bulk viscosity. Also as in classical fluids, heat conduction is present
in relativistic theories of non-perfect fluids. Since these are all phenomena associated with 
out-of-equilibrium physics, for simplicity we will henceforth refer to all of them simply as ``viscous,"
making no distinction between viscosity and heat conduction effects.}
(but with no shear viscosity nor heat conduction) is locally well-posed 
and causal in the full nonlinear\footnote{Talking about causality in the ``nonlinear
regime" is redundant in that the equations of motion are nonlinear. However, this language
is sometimes used in the literature to make a contrast with earlier and more common
causality results that apply only to the linearization of the equations about constant solutions.} 
regime, both in a fixed background and when the
equations are coupled to Einstein's equations \cite{DisconziBemficaNoronhaISBulk}
(see \cite{Pu:2009fj,Floerchinger:2017cii,Denicol:2008ha} for earlier causality results also valid 
in the nonlinear regime but under strong symmetry assumptions or in $1+1$ dimensions).
A similar causality and well-posedness result is valid in Gevrey spaces when shear viscosity is 
present \cite{DisconziBemficaHoangNoronhaRadoszNonlinearConstraints}.

Its great success nonetheless, it is far from clear whether the M\"uller-Israel-Stewart theory
provides the most accurate description of relativistic viscous fluids over all scales
where the fluid approximation is supposed to hold and viscous effects expected to be
relevant. For instance, it is not known whether the M\"uller-Israel-Stewart 
equations can be applied to the study of neutron star mergers 
\cite{Alford:2017rxf,Most:2018eaw}.
Moreover, the mathematical foundations of the M\"uller-Israel-Stewart are for the most part
lacking, with the aforementioned results \cite{DisconziBemficaNoronhaISBulk,DisconziBemficaHoangNoronhaRadoszNonlinearConstraints} being the only 
ones available in the literature. Finally, the M\"uller-Israel-Stewart equations do not
seem capable to describe the dynamics of shock waves 
or more general types of fluid singularities \cite{Olson:1991pf,GerochLindblom-1991,DisconziHoangRadosz}. In view of these limitations,
there is a strong interest in searching for alternative theories of relativistic viscous fluids
\cite{Kovtun:2019hdm}.

The instability results that ruled out the Eckart and Landau theories are in fact applicable
to a large class of first-order theories \cite{Hiscock:1985zz}. Consequently, for a long time it 
was thought that first-order theories were intrinsically unstable (see discussions in
\cite{RezzollaZanottiBookRelHydro-2013,Kovtun:2019hdm,Tsumura:2007ji, Tsumura:2011cj,Tsumura:2012ss,Van:2007pw,Van:2011yn}).
Nevertheless, in recent years this perception has been shown to be overstated, with 
several different results showing the viability of first-order theories. In \cite{DisconziViscousFluidsNonlinearity},
causality and local well-posedness (in Gevrey spaces) of the Lichnerowicz theory has been established in the case
of irrotational fluids with or without coupling to Einstein's equations,
a result that has been slightly improved in \cite{DisconziCzubakNonzero}.
The Lichnerowicz theory is a first-order theory introduced in \cite{LichnerowiczBookGR-1955} and which has 
led to interesting applications in cosmology 
\cite{DisconziKephartScherrerNewApproach,DisconziKephartScherrerViableFirstOrder}. However, it remains
open whether Lichnerowicz's theory is stable. In \cite{DisconziBemficaNoronhaConformal}
a first-order theory of relativistic viscous \emph{conformal} fluids has been introduced based on
kinetic theory. Its stability, causality, and local well-posedness (in Gevrey spaces when the equations
are coupled to Einstein's equations and in Sobolev
spaces when the equations are considered in Minkowski background) has been proven in the works
\cite{DisconziBemficaNoronhaConformal,DisconziExistenceCausalityConformal
,DisconziBemficaRodriguezShaoSobolevConformal} , and applications
relevant to the study of the quark-gluon plasma have also been developed \cite{DisconziBemficaNoronhaConformal}. 
In \cite{FreistuhlerTemple-2014,FreistuhlerTemple-2017,FreistuhlerTemple-2018}
(see also \cite{Freistuhler-2020}) a first-order theory has been
introduced for which stability holds in the fluid's rest frame. This leads to the possibility
that such theory might be stable and causal, although it is known that stability in the rest frame
is not enough to ensure stability in general \cite{Hiscock:1985zz}. 
Earlier first-order theories for which stability
has also been established can be found in \cite{Van:2007pw,Van:2011yn}.
Aside from all these results concerning first-order theories, further causality, stability, 
and local-wellposedness results have been established
in the context of the so-called divergence-type theories
\cite{Lehner:2017yes,GerochLindblom-DivergenceType-1990, GerochLindblom-1991,
Liu-Muller-Ruggeri-1986,Muller-Ruggeri-Book-1998,PeraltaRamos:2009kg,PeraltaRamos:2010je},
which constitute examples of second-order theories
different than the M\"uller-Israel-Stewart theory (see also 
\cite{Nagy-Ortiz-Reula-1994,Kreiss-Nagy-Ortiz-Reula-1997,Reula-Nagy-1997}).
We also mention the so-called anisotropic hydrodynamics
\cite{Alqahtani:2017mhy}, which is a stable second-order theory that has been 
very successful in studies of the quark-gluon plasma,
although to the best of our knowledge there has been no results showing causality
or local well-posedenss for anisotropic hydrodynamics.

The previous discussion highlights not only the importance of investigating relativistic fluids
with viscosity but also how its study is a very active field of research, with some of the
most basic questions, namely, causality, stability, and local well-posedness, remaining
largely open. This paper is concerned with the well-posedness of the Cauchy 
problem for the first-order theory of relativistic fluids defined by the energy-momentum
tensor \eqref{E:Energy_momentum} below. 

This energy-momentum tensor was introduced simultaneously in \cite{DisconziBemficaNoronhaBarotropic}
and  \cite{Kovtun:2019hdm} using effective field theory arguments. In \cite{DisconziBemficaNoronhaBarotropic}
a kinetic theory derivation (at zero chemical potential) was also given, while
\cite{Hoult:2020eho} discussed the necessary modifications that stem from the inclusion of a conserved current.
Under the
assumption of a barotropic equation of state (i.e., when the pressure is a function
of the energy density only), the 
stability of the corresponding equations of motion has been established in
these works, whereas in \cite{DisconziBemficaNoronhaBarotropic} causality
and local well-posedness of the equations of motion has also been proven. Such
local well-posedness has been established in Gevrey spaces 
with and without coupling to Einstein's equations\footnote{In fact, a slightly weaker 
statement has been proved in \cite{DisconziBemficaNoronhaBarotropic}, but this
does not change the overall theme discussed here nor the goal of this manuscript. See
Remark \ref{R:Gevrey_original}.}. Our goal in this manuscript is to improve
this result by proving local well-posedness in Sobolev spaces. However, contrary to
\cite{DisconziBemficaNoronhaBarotropic}, here we do not consider coupling
to Einstein's equations, restricting ourselves to the case where the evolution takes place
in Minkowski space.

We finish this introductions with two explanatory remarks. First, the question of the correct
theory of relativistic viscous fluids cannot be decided solely by considerations from microscopic theory.
This is because the same underlying microscopic theory can give rise to different, inequivalent, 
fluid approximations depending on the chosen coarse-graining procedure 
\cite{Denicol:2012cn,DeGroot:1980dk}. 
Second, above
we referred to numerical simulations that show the importance of viscous effects in
neutron star mergers \cite{Alford:2017rxf}. We remark that these
simulations do not numerically solve models relativistic fluids with viscosity, 
relying rather on estimates for the relevant transport scales and the size of gradients of the hydrodynamic fields determined in an inviscid evolution.
Indeed, as hinted above, it is not yet known which, if any, of the current models of relativistic viscous fluids
is appropriate to describe neutron star mergers.

\medskip
\noindent \textbf{Acknowledgments:} We are grateful to Magdalena Czubak for discussions.

\section{Equations of motion and statement of the results}

The energy-momentum tensor that defines the first-order theory of relativistic viscous fluids
studied here (introduced in \cite{DisconziBemficaNoronhaBarotropic,Kovtun:2019hdm},
see discussion in the introduction) is given by  
\begin{align}
\label{E:Energy_momentum}
\mathcal{T}_{\alpha \beta} & = 
(\varepsilon + A_1 )u_\alpha u_\beta  + 
(P + A_2) \Proj_{\alpha \beta}
- 2 \upeta \upsigma_{\alpha \beta} + u_\alpha Q_\beta + u_\beta Q_\alpha,
\end{align}
where
\begin{align}
\begin{split}
A_1 & 
= \upchi_1 \frac{u^\alpha \nabla_\alpha \varepsilon}{\varepsilon + P} 
+ \upchi_2 \nabla_\alpha u^\alpha,
\\
A_2 & = 
\upchi_3 \frac{u^\alpha \nabla_\alpha \varepsilon}{\varepsilon + P}
+ \upchi_4 \nabla_\alpha u^\alpha,
\\
Q_\alpha & = \uplambda  ( \frac{c_s^2}{\varepsilon + P} \Proj_{\alpha}^\mu  \nabla_\mu \varepsilon + 
u^\mu \nabla_\mu u_\alpha ),
\\
\upsigma_{\alpha\beta} & = \frac{1}{2}( \Proj_\alpha^\mu \nabla_\mu u_\beta + \Proj_\beta^\mu \nabla_\mu u_\alpha
-\frac{2}{3} \Proj_{\alpha\beta} \nabla_\mu u^\mu).
\end{split}
\nonumber
\end{align}
Here, $\varepsilon$ is the fluid's energy density; $P$ is the fluid's pressure, where we assume
a barotropic equation of state, thus $P=P(\varepsilon)$;
$g$ is the spacetime metric\footnote{By ``metric" we always mean a ``Lorentzian
metric."}; $u$ is the fluid's four-velocity, which is future-pointing and unit timelike with respect 
to $g$, so in particular $u$ satisfies the constraint
\begin{align}
\label{E:u_unit}
g_{\alpha\beta} u^\alpha u^\beta &= -1.
\end{align}
Notice that we are assuming the spacetime to be time-oriented as $u$ is taken as a future-pointing
vectorfield. In practice, we will work in Minkowski space with standard orientation;
 $\Proj$ is the projection onto the space orthogonal to $u$, given by
\begin{align}
\Proj_{\alpha\beta} = g_{\alpha\beta}+ u_\alpha u_\beta;
\nonumber
\end{align}
$\upeta$, $\upchi_1,\upchi_2, \upchi_3,\upchi_4$, and $\uplambda$ are
 transport coefficients, which are known functions of $\varepsilon$ and 
 model the viscous effects in the fluid; and
 $\nabla$ is the covariant derivative associated with the metric $g$.
Indices are raised and lowered using the spacetime metric. We adopt the convention that
lowercase Greek indices vary from $0$
to $3$, Latin indices vary from $1$ to $3$, and
repeated indices are summed over their range. Expressions such as
$z_{\alpha}$, $w_{\alpha\beta}$, etc. represent the components of a vector or tensor
with respect to a system of coordinates $\{ x^\alpha \}_{\alpha=0}^3$ in spacetime,
where the coordinates are always chosen so that $x^0=t$ represents a time coordinate.
We will consider the fluid dynamics in Minkowski background, so that the $g$ is the Minkowski metric.
We note for future reference that equation \eqref{E:u_unit} implies
\begin{align}
u^\alpha \nabla_\beta u_\alpha = 0.
\label{E:Derivative_u_orthogonal}
\end{align}

The equations of motion are given by 
\begin{align}
\label{E:Div_T}
\nabla_\alpha \mathcal{T}^\alpha_\beta = 0
\end{align}
supplemented by the constraint \eqref{E:u_unit}. 

We are now ready to state our main result, which is the following.

\begin{theorem}
\label{T:main_theorem}
Let $g$ be the Minkowski metric on $\mathbb{R} \times \mathbb{T}^3$,
where $\mathbb{T}^3$ is the three-dimensional torus. Let $P,\upeta, \upchi_1,\upchi_2,\upchi_3,\upchi_4, \uplambda: (0,\infty) \rightarrow (0,\infty)$ be analytic functions satisfying
$\uplambda,\upchi_1,\upeta>0$, $P > 0$, $c_s^2 := P^\prime > 0$, and
\begin{align}
& 9 \uplambda ^2 \upchi _2^2 c_s^4+6 \uplambda  c_s^2 \left(\upchi _1 \left(4 \upeta -3 \upchi _4\right) \left(2 \uplambda +\upchi _2\right)+3 \upchi _2 \upchi _3 \left(\uplambda +\upchi_2\right)\right) \\
&+\left(\upchi _1 \left(4 \upeta -3 \upchi _4\right)+3 \upchi _3 \left(\uplambda +\upchi_2\right)\right){}^2>0,\nonumber
\\
&\uplambda\ge \upeta,
\nonumber
\\
&3\upchi_4> 4\upeta,
\nonumber \\
&2\uplambda\upchi_1\ge \uplambda\upchi_2c_s^2-\upchi_1\left (\upchi_4-\frac{4\upeta}{3}\right )+\uplambda\upchi_3+\upchi_3\upchi_2,\nonumber
\\
&\uplambda\upchi_1+c_s^2\uplambda\left (\upchi_4-\frac{4\upeta}{3}\right )\ge c_s^2\uplambda \upchi_2 + \uplambda\upchi_3 + \upchi_2 \upchi_3 - \upchi_1 \left (\upchi_4-\frac{4}{3} \upeta\right ) \ge 0.
\nonumber
\end{align}
Let $\varepsilon_{(0)} \in H^r(\mathbb{T}^3,\mathbb{R})$,
$\varepsilon_{(1)} \in H^{r-1}(\mathbb{T}^3,\mathbb{R})$,
$u_{(0)} \in H^r(\mathbb{T}^3,\mathbb{R}^3)$, 
and 
$u_{(1)} \in H^{r-1}(\mathbb{T}^3,\mathbb{R}^3)$
be given, where $H^r$ is the Sobolev space and $r > 9/2$. Assume that 
$\varepsilon_{(0)} \geq C_0 > 0$ for some constant $C_0$.

Then, there exists a $T>0$, a function 
\begin{align}
\nonumber
\varepsilon \in C^0([0,T],H^r(\mathbb{T}^3,\mathbb{R})) \cap 
C^1([0,T],H^{r-1}(\mathbb{T}^3,\mathbb{R})) \cap 
C^2([0,T],H^{r-2}(\mathbb{T}^3,\mathbb{R})),
\end{align}
and a vector field 
\begin{align}
u \in C^0([0,T],H^r(\mathbb{T}^3,\mathbb{R}^4)) \cap 
C^1([0,T],H^{r-1}(\mathbb{T}^3,\mathbb{R}^4)) \cap 
C^2([0,T],H^{r-2}(\mathbb{T}^3,\mathbb{R}^4))
\end{align}
such that equations \eqref{E:u_unit} and \eqref{E:Div_T} hold on 
$[0,T]\times \mathbb{T}^3$, and satisfy
$\varepsilon(0,\cdot) = \varepsilon_{(0)}$, 
$\partial_t \varepsilon(0,\cdot) = \varepsilon_{(1)}$,
$\mathcal{P}u(0,\cdot) = u_{(0)}$,
and $\mathcal{P} \partial_t u(0,\cdot) = u_{(1)}$, where $\partial_t$
is the derivative with respect to the first coordinate in $[0,T]\times \mathbb{T}^3$
and $\mathcal{P}$ is the canonical projection from the tangent bundle of
$[0,T]\times \mathbb{T}^3$ onto the tangent bundle of $\mathbb{T}^3$.
Moreover, $(\varepsilon,u)$ is the unique solution with the stated properties.
\end{theorem}

We proceed to make some comments about the assumptions and conclusions of Theorem \ref{T:main_theorem}.

We note that in view of \eqref{E:u_unit}, 
it suffices to provide the components of $u$ tangent to $\{t=0\}$
as initial data; this explains the statement involving the projector $\mathcal{P}$
in the Theorem. On the other hand, throughout the manuscript, we will consider systems of equations
for the full four-velocity $u = (u^0,u^1,u^2,u^3)$. In these cases, we will always take the initial condition
for $u$ defined by \eqref{E:u_unit} and \eqref{E:Derivative_u_orthogonal} when $(u^1,u^2,u^3)$ takes
the values of the given initial data.

The quantity $c_s^2$ corresponds to the fluid's sound speed in 
the case of a perfect fluid. In the presence of viscosity, the fluid's sound speed
is no longer given by $c_s^2$ (see section \ref{S:Characteristics} for a description
of the characteristic speeds of the system), but it is still convenient to introduce $c_s^2$.
We work on $\mathbb{T}^3$ for simplicity, since using the domain of dependence
property (proved in \cite{DisconziBemficaNoronhaBarotropic}) one can adapt the proof
to $\mathbb{R}^3$. On the other hand, the assumption $\varepsilon_0 \geq C_0 > 0$ is essential.
The equations can otherwise degenerate, resulting in a free-boundary
dynamics, a problem that only quite recently was solved for the case of a perfect fluid
\cite{DisconziIfrimTataru,Oliynyk-2019-arxiv,Miao-Shahshahani-Wu-2020-arxiv}
(see \cite{Hadzic-Shkoller-Speck-2019,Jang-LeFloch-Masmoudi-2016,Ginsberg-2019,
Oliynyk-2012,Trakhinin-2009} for earlier work focusing on particular cases or a priori estimates).

The assumptions on $P,\upeta, \upchi_1,\upchi_2,\upchi_3,\upchi_4$ and $\uplambda$ in Theorem 
\ref{T:main_theorem} are precisely the conditions found in 
\cite{DisconziBemficaNoronhaBarotropic} that ensure the causality and stability of the equations of motion. Although these conditions are a bit cumbersome to write, it is not difficult to see that they
are not empty. Moreover, given a specific choice of equation of state and transport coefficients,
it is generally not difficult to verify whether such conditions are satisfied.

The idea behind the proof of Theorem \ref{T:main_theorem} can be summarized as follows. First, we use
\eqref{E:u_unit} and \eqref{E:Derivative_u_orthogonal} to decompose \eqref{E:Div_T}
in the directions parallel and orthogonal to $u$, as it is customary in both the cases of perfect
and non-perfect relativistic fluids. Next, we construct new variables out of certain
combinations of $u$, $\varepsilon$, and its derivatives, and rewrite the equations of motion
in terms of these new variables. We then show that, under the hypotheses of the Theorem,
the principal symbol of the new system of equations can be diagonalized. This diagonalization
procedure can be carried over to the equations of motion upon the introduction of suitable
pseudodifferential operators. The pseudifferential calculus is needed because 
the diagonalization of the principal symbol involves certain rational functions of the eigenvalues 
and of the determinant of the principal symbol. Due to the quasilinear nature of the problem, we have
to deal with symbols of limited smoothness. Nevertheless, we are still able to obtain
good energy estimates for a linearized version of the problem that can be used to set up
a convergent iteration scheme, leading to existence and uniqueness of solutions to the 
new equations of motion we introduced. At this point, we need to show that this result
gives rise to existence and uniqueness of solutions to the original equations of motion, i.e.,
\eqref{E:u_unit} and \eqref{E:Div_T}. To do so, we need to derive yet another system
of equations that ensures that the constraint \eqref{E:Div_T} is satisfied. For this new system,
solutions are obtained in a more restrictive class of functions, namely, Gevrey functions. 
Since these are dense in Sobolev spaces, we finally obtain existence and uniqueness for the original
problem, in Sobolev spaces, by an approximation argument.

\section{A new system of equations\label{S:New_system}}
In this section we derive a new system of equations that will allow
us to establish Theorem \ref{T:main_theorem}. In order to do so, throughout this section,
we assume to be given a sufficiently regular solution to equations
\eqref{E:u_unit} and \eqref{E:Div_T}.

We begin using \eqref{E:u_unit} to decompose $\nabla_\alpha \mathcal{T}^\alpha_\beta$
in the directions orthogonal  and parallel to $u$, so that equation \eqref{E:Div_T}
gives
\begin{subequations}{\label{E:EofM}}
\begin{align}
&
u^\alpha\nabla_\alpha A_1+\nabla_\alpha Q^\alpha+(\varepsilon+P+A_1+A_2)\nabla_\alpha u^\alpha+u^\alpha\nabla_\alpha \varepsilon-2\upeta\sigma^{\alpha\beta}\sigma_{\alpha\beta}\\
&+Q_\alpha u^\beta\nabla_\beta u^\alpha=0,
\label{E:EofM_1}
\\
& 
\Proj^{\alpha\beta}\nabla_\beta A_2+u^\beta\nabla_\beta Q^\alpha-2\upeta\nabla_\beta\sigma^{\alpha\beta}+(\varepsilon+P+A_1+A_2)u^\beta\nabla_\beta u^\alpha+c_s^2\Proj^{\alpha\beta}\nabla_\beta\varepsilon
-2\sigma^{\alpha\beta}\nabla_\beta\upeta\\
&+2u^\alpha\upeta\sigma^{\mu\nu}\sigma_{\mu\nu}+Q^\beta\nabla_\beta u^\alpha-u^\alpha Q^\beta u^\mu\nabla_\mu u_\beta+Q^\alpha \nabla_\beta u^\beta=0.
\label{E:EofM_2}
\end{align}
\end{subequations}
Define 
\begin{align}
\begin{split}
S_\alpha^{\mss \beta} & := \Proj_\alpha^\mu \nabla_\mu u^\beta,
\\
\mathsf{S}^\alpha & := u^\mu \nabla_\mu u^\alpha,\\
\mathcal{V} & :=\frac{u^\mu \nabla_\mu \varepsilon}{\varepsilon+P},\\
V^\alpha & :=\frac{\Proj^{\alpha\beta} \nabla_\beta \varepsilon}{\varepsilon+P},
\end{split}
\nonumber
\end{align}
so that
\begin{subequations}{\label{E:Equations_1st_order}}
\begin{align}
\upchi_1 u^\alpha \nabla_\alpha \mathcal{V} +\uplambda c_s^2\nabla_\alpha V^\alpha+\uplambda \nabla_\alpha \mathsf{S}^\alpha+\upchi_2 u^\alpha\nabla_\alpha S^{\mss\nu}_\nu + r_1 & = 0,
\label{E:Equations_1st_order_1}
\\
\upchi_3\Proj^{\mu\alpha}\nabla_\alpha \mathcal{V} +\uplambda c_s^2 u^\alpha \nabla_\alpha V^\mu+\uplambda u^\alpha\nabla_\alpha \mathsf{S}^\mu 
+ B_\nu^{\mss\mu\lambda\alpha}\nabla_\alpha S_\lambda^{\mss \nu}
+ r_2
& =0,
\label{E:Equations_1st_order_2}
\\
-\Proj^{ \mu\alpha}\nabla_\alpha \mathcal{V} + u^\alpha \nabla_\alpha V^\mu +r_3& = 0,
\label{E:Equations_1st_order_3}
\\
u^\alpha \nabla_\alpha S_\lambda^{\mss \nu} - \Proj_\lambda^\alpha \nabla_\alpha
\mathsf{S}^\nu + r_4 & = 0,
\label{E:Equations_1st_order_4}
\\
u^\alpha \nabla_\alpha \varepsilon + r_5 & = 0,
\label{E:Equations_1st_order_5}
\\
u^\mu \nabla_\mu u^\alpha+ r_6 & = 0,
\label{E:Equations_1st_order_6}
\end{align}
\end{subequations}
where 
\begin{align}
B_\nu^{\mss\mu\lambda\alpha} & := (\upchi_4+\frac{2\upeta}{3})\Proj^{\mu\alpha}\delta_\nu^\lambda-\upeta(\Proj^{\alpha\lambda}\delta^\mu_\nu+\Proj^{\mu\lambda}\delta^\alpha_\nu).
\nonumber
\end{align}
Above, $r_i$, $i=1,\dots,6$ are analytic functions of 
$\mathcal{V}$, $V^\nu$, $\mathsf{S}^\nu$, $S_\lambda^{\mss \nu}$, 
$\varepsilon$, and $u^\alpha$; no derivative of these quantities appears in the 
$r_i$'s.

We now provide details on the derivation of \eqref{E:Equations_1st_order}.
Equations \eqref{E:Equations_1st_order_1} and 
\eqref{E:Equations_1st_order_2}
are equations \eqref{E:EofM_1} and \eqref{E:EofM_2}, respectively;
equations
\eqref{E:Equations_1st_order_5} and \eqref{E:Equations_1st_order_6}
are simply the definition of $V$ and $\mathsf{S}^\nu$;
equations \eqref{E:Equations_1st_order_3} and 
\eqref{E:Equations_1st_order_4} follow from contracting the identities
\begin{align}
\nabla_\mu \nabla_\nu \varepsilon - \nabla_\nu \nabla_\mu \varepsilon & = 0,\\
\nabla_\mu \nabla_\nu u^\alpha -\nabla_\nu \nabla_\mu u^\alpha & = 
R_{\mu \nu \ms \lambda}^{\mss \mss \alpha} u^\lambda = 0,
\end{align}
with $u^\mu$ and then with $\Proj^\nu_\lambda$. We also used the identity
\begin{align}
\nabla_\alpha u^\beta & = -u_\alpha \mathsf{S}^\beta + S_\alpha^{\mss \beta}.
\end{align} 
We write equations \eqref{E:Equations_1st_order} as a quasilinear first order system
for the variable\\ $\mathbf{\Psi} = (\mathcal{V}, V^\nu,\mathsf{S}^\nu, 
S_0^{\mss \nu},S_1^{\mss \nu},S_2^{\mss \nu},S_3^{\mss \nu},
\varepsilon, u^\nu)$ as
\begin{align}
\mathcal{A}^\alpha \nabla_\alpha \mathbf{\Psi} + \mathcal{R} & = 0,
\label{E:Matrix_system_1st_order}
\end{align}
where $\mathcal{R} = (r_1,\dots,r_6)$ and
$\mathcal{A}^\alpha$ is given by
\begin{align}
\mathcal{A}^\alpha=
\begin{bmatrix}
\upchi_1 u^\alpha & \uplambda c_s^2 \updelta^\alpha_\nu & \uplambda\updelta^\alpha_\nu & \upchi_2 u^\alpha\updelta^0_\nu & \upchi_2 u^\alpha\updelta^1_\nu & \upchi_2 u^\alpha\updelta^2_\nu & \upchi_2 u^\alpha\updelta^3_\nu & 0 & 0_{1\times 4}\\
\upchi_3\Proj^{\mu\alpha} & \uplambda c_s^2 u^\alpha I_4 & \uplambda u^\alpha I_4 & B_\nu^{\mss\mu 0 \alpha} & 
B_\nu^{\mss\mu 1 \alpha} & B_\nu^{\mss\mu 2 \alpha} & B_\nu^{\mss\mu 3 \alpha} & 0_{4\times 1} & 0_{4\times 4} \\
-\Proj^{\mu\alpha} & u^\alpha I_4 & 0_{4\times 4} & 0_{4\times 4} & 0_{4\times 4} & 0_{4\times 4} & 0_{4\times 4} & 0_{4\times 1} & 0_{4\times 4} \\
0_{4\times 1} & 0_{4\times 4} & -\Proj^\alpha_0 I_4 & u^\alpha I_4 & 0_{4\times 4} & 0_{4\times 4} & 0_{4\times 4} & 0_{4\times 1} & 0_{4\times 4} \\
0_{4\times 1} & 0_{4\times 4} & -\Proj^\alpha_1 I_4 & 0_{4\times 4} & u^\alpha I_4 & 0_{4\times 4} & 0_{4\times 4} & 0_{4\times 1} & 0_{4\times 4} \\
0_{4\times 1} & 0_{4\times 4} & -\Proj^\alpha_2 I_4 & 0_{4\times 4} & 0_{4\times 4} & u^\alpha I_4 & 0_{4\times 4} & 0_{4\times 1} & 0_{4\times 4} \\
0_{4\times 1} & 0_{4\times 4} & -\Proj^\alpha_3 I_4 & 0_{4\times 4} & 0_{4\times 4} & 0_{4\times 4} & u^\alpha I_4 & 0_{4\times 1} & 0_{4\times 4} \\
0 & 0_{1\times 4} &  0_{1\times 4}  & 0_{1\times 4} & 0_{1\times 4} & 0_{1\times 4} &  0_{1\times 4}  & u^\alpha & 0_{1\times 4} \\
0_{4\times 1} & 0_{4\times 4} & 0_{4\times 4} & 0_{4\times 4} & 0_{4\times 4} & 0_{4\times 4} & 0_{4\times 4} & 0_{4\times 1} & u^\alpha I_4 \\
\end{bmatrix},
\nonumber
\end{align}
where $0_{m\times m}$ is the $m\times m$ zero matrix and $I_{m\times m}$ is the $m \times m$
identity matrix.
Equation \eqref{E:Matrix_system_1st_order} is the main equation we will use to derive estimates.

\section{Diagonalization\label{S:Diagonalization}}
For everything that follows, we work under the assumptions of Theorem \ref{T:main_theorem}.

\begin{remark}[Silent use of \eqref{E:u_unit} and \eqref{E:Derivative_u_orthogonal}]
Throughout our computations, we will make successive use of 
equations \eqref{E:u_unit} and \eqref{E:Derivative_u_orthogonal} without explicitly 
mentioning them.
\end{remark}

\begin{proposition}
\label{P:Diagonalization}
 Let $\upxi$ be a timelike vector and assume that $\uplambda,\upchi_1,\upeta>0$, and
\begin{align}
&\Delta_D := 9 \uplambda ^2 \upchi _2^2 c_s^4+6 \uplambda  c_s^2 \left(\upchi _1 \left(4 \upeta -3 \upchi _4\right) \left(2 \uplambda +\upchi _2\right)+3 \upchi _2 \upchi _3 \left(\uplambda +\upchi_2\right)\right) \\
&+\left(\upchi _1 \left(4 \upeta -3 \upchi _4\right)+3 \upchi _3 \left(\uplambda +\upchi_2\right)\right){}^2>0,\label{Delta}\\
&\uplambda\ge \upeta,\label{10a}\\
&3\upchi_4> 4\upeta,\label{10b}\\
&2\uplambda\upchi_1\ge \uplambda\upchi_2c_s^2-\upchi_1\left (\upchi_4-\frac{4\upeta}{3}\right )+\uplambda\upchi_3+\upchi_3\upchi_2,\label{10b2}\\
&\uplambda\upchi_1+c_s^2\uplambda\left (\upchi_4-\frac{4\upeta}{3}\right )\ge c_s^2\uplambda \upchi_2 + \uplambda\upchi_3 + \upchi_2 \upchi_3 - \upchi_1 \left (\upchi_4-\frac{4}{3} \upeta\right ) \ge 0.\label{10c}
\end{align}
Then:\\
\noindent (i) $\det(\mathcal{A}^\alpha \upxi_\alpha) \neq 0$;\\
\noindent (ii) For any spacelike vector $\upzeta$, the eigenvalue problem 
$\mathcal{A}^\alpha(\upzeta_\alpha + \Lambda \upxi_\alpha)  V = 0$
has only real eigenvalues $\Lambda$ and a complete set of eigenvectors $V$.
\end{proposition}
\begin{proof}
Let $\Xi_\alpha$ be any co-vector and $\mathfrak{a} :=u^\alpha\Xi_\alpha$, $\mathfrak{b}^\alpha :=\Proj^{\alpha\beta}\Xi_\beta$, and $D_\nu^{\mss\mu\lambda} :=B_\nu^{\mss\mu\lambda\alpha}\Xi_\alpha$. Also, consider the superscript $\mu$ labeling rows while the subscript $\nu$ labels columns. Then
\begin{align}
\det(\Xi_\alpha \mathcal{A}^\alpha) =&
\det
\begin{bmatrix}
\upchi_1 \mathfrak{a} & \uplambda c_s^2 \Xi_\nu & \uplambda\Xi_\nu & \upchi_2 \mathfrak{a}\updelta^0_\nu & \upchi_2 \mathfrak{a}\updelta^1_\nu & \upchi_2 \mathfrak{a}\updelta^2_\nu & \upchi_2 \mathfrak{a}\updelta^3_\nu & 0 & 0_{1\times 4}\\
\upchi_3\mathfrak{b}^\mu & \uplambda c_s^2 \mathfrak{a} I_4 & \uplambda \mathfrak{a} I_4 & D_\nu^{\mss\mu 0} & 
D_\nu^{\mss\mu 1} & D_\nu^{\mss\mu 2 } & D_\nu^{\mss\mu 3} & 0_{4\times 1} & 0_{4\times 4} \\
-\mathfrak{b}^\mu & \mathfrak{a} I_4 & 0_{4\times 4} & 0_{4\times 4} & 0_{4\times 4} & 0_{4\times 4} & 0_{4\times 4} & 0_{4\times 1} & 0_{4\times 4} \\
0_{4\times 1} & 0_{4\times 4} & -\mathfrak{b}_0 I_4 & \mathfrak{a} I_4 & 0_{4\times 4} & 0_{4\times 4} & 0_{4\times 4} & 0_{4\times 1} & 0_{4\times 4} \\
0_{4\times 1} & 0_{4\times 4} & -\mathfrak{b}_1 I_4 & 0_{4\times 4} & \mathfrak{a} I_4 & 0_{4\times 4} & 0_{4\times 4} & 0_{4\times 1} & 0_{4\times 4} \\
0_{4\times 1} & 0_{4\times 4} & -\mathfrak{b}_2 I_4 & 0_{4\times 4} & 0_{4\times 4} & \mathfrak{a} I_4 & 0_{4\times 4} & 0_{4\times 1} & 0_{4\times 4} \\
0_{4\times 1} & 0_{4\times 4} & -\mathfrak{b}_3 I_4 & 0_{4\times 4} & 0_{4\times 4} & 0_{4\times 4} & \mathfrak{a} I_4 & 0_{4\times 1} & 0_{4\times 4} \\
0 & 0_{1\times 4} &  0_{1\times 4}  & 0_{1\times 4} & 0_{1\times 4} & 0_{1\times 4} &  0_{1\times 4}  & \mathfrak{a} & 0_{1\times 4} \\
0_{4\times 1} & 0_{4\times 4} & 0_{4\times 4} & 0_{4\times 4} & 0_{4\times 4} & 0_{4\times 4} & 0_{4\times 4} & 0_{4\times 1} & \mathfrak{a} I_4 \\
\end{bmatrix}\\
 =&
\mathfrak{a}^{5}\det
\begin{bmatrix}
\upchi_1 \mathfrak{a} & \uplambda c_s^2 \Xi_\nu & \uplambda\Xi_\nu & \upchi_2 \mathfrak{a}\updelta^0_\nu & \upchi_2 \mathfrak{a}\updelta^1_\nu & \upchi_2 \mathfrak{a}\updelta^2_\nu & \upchi_2 \mathfrak{a}\updelta^3_\nu\\
\upchi_3\mathfrak{b}^\mu & \uplambda c_s^2 \mathfrak{a} I_4 & \uplambda \mathfrak{a} I_4 & D_\nu^{\mss\mu 0} & 
D_\nu^{\mss\mu 1} & D_\nu^{\mss\mu 2 } & D_\nu^{\mss\mu 3} \\
-\mathfrak{b}^\mu & \mathfrak{a} I_4 & 0_{4\times 4} & 0_{4\times 4} & 0_{4\times 4} & 0_{4\times 4} & 0_{4\times 4}  \\
0_{4\times 1} & 0_{4\times 4} & -\mathfrak{b}_0 I_4 & \mathfrak{a} I_4 & 0_{4\times 4} & 0_{4\times 4} & 0_{4\times 4} \\
0_{4\times 1} & 0_{4\times 4} & -\mathfrak{b}_1 I_4 & 0_{4\times 4} & \mathfrak{a} I_4 & 0_{4\times 4} & 0_{4\times 4} \\
0_{4\times 1} & 0_{4\times 4} & -\mathfrak{b}_2 I_4 & 0_{4\times 4} & 0_{4\times 4} & \mathfrak{a} I_4 & 0_{4\times 4} \\
0_{4\times 1} & 0_{4\times 4} & -\mathfrak{b}_3 I_4 & 0_{4\times 4} & 0_{4\times 4} & 0_{4\times 4} & \mathfrak{a} I_4 \\
\end{bmatrix}
\\
=&
\mathfrak{a}^{17}\det
\begin{bmatrix}
\upchi_1 \mathfrak{a} & \uplambda c_s^2 \Xi_\nu & \uplambda\Xi_\nu+\upchi_2\mathfrak{b}_\nu \\
\upchi_3\mathfrak{a}\mathfrak{b}^\mu & \uplambda c_s^2 \mathfrak{a}^2 I_4 & \uplambda \mathfrak{a}^2 I_4+ D_\nu^{\mss \mu \lambda}\mathfrak{b}_\lambda \\
-\mathfrak{b}^\mu & \mathfrak{a} I_4 & 0_{4\times 4}  \\
\end{bmatrix}
\\
= &
\frac{\mathfrak{a}^{14}}{\upchi_1^3}\det
\begin{bmatrix}
(\uplambda c_s^2+\upchi_3 )\mathfrak{a}^2I_4 & \uplambda \mathfrak{a}^2 I_4+ D_\nu^{\mss \mu \lambda}\mathfrak{b}_\lambda \\
\mathfrak{a}^2\upchi_1 I_4+\uplambda c_s^2 \Xi_\nu \mathfrak{b}^\mu & (\uplambda\Xi_\nu+\upchi_2\mathfrak{b}_\nu)\mathfrak{b}^\mu  \\
\end{bmatrix}
\\
=&
\frac{\mathfrak{a}^{14}}{\upchi_1^3}\det((\uplambda c_s^2+\upchi_3)\mathfrak{a}^2(\uplambda\Xi_\nu+\upchi_2\mathfrak{b}_\nu)\mathfrak{b}^\mu
-(\mathfrak{a}^2\upchi_1 \delta^\mu_\sigma+\uplambda c_s^2 \Xi_\sigma \mathfrak{b}^\mu)(\uplambda \mathfrak{a}^2 \delta_\nu^\sigma+ D_\nu^{\mss \sigma \lambda}\mathfrak{b}_\lambda))
\\
=&
\frac{\mathfrak{a}^{14}}{\upchi_1^3}\det(\upchi_1\mathfrak{a}^2(\uplambda\mathfrak{a}^2-\upeta\mathfrak{b}^\alpha\mathfrak{b}_\alpha)I_4+
((\upchi_4+\frac{2\upeta}{3})(\mathfrak{a}^2\upchi_1+\uplambda c_s^2\mathfrak{b}^\alpha\mathfrak{b}_\alpha)-(\uplambda c_s^2+\upchi_3)\upchi_2\mathfrak{a}^2)\mathfrak{b}^\mu\mathfrak{b}_\nu\\
&-(\mathfrak{a}^2(\upchi_1\upeta+\upchi_3\uplambda)+2\uplambda\upeta c_s^2 \mathfrak{b}^\alpha\mathfrak{b}_\alpha)\mathfrak{b}^\mu\Xi_\nu)
\\
=&\mathfrak{a}^{20}(\uplambda\mathfrak{a}^2-\upeta\mathfrak{b}^\alpha\mathfrak{b}_\alpha)^3(\lambda (\upchi_1 \mathfrak{a}^2-\upchi_3\mathfrak{b}^\alpha\mathfrak{b}_\alpha)\mathfrak{a}^2-\upchi_2(\upchi_3+\uplambda c_s^2 ) \mathfrak{a}^2\mathfrak{b}^\alpha\mathfrak{b}_\alpha\\
&+(\upchi_1 \mathfrak{a}^2+\uplambda c_s^2)(\upchi_4-\frac{4\upeta}{3}\mathfrak{b}^\beta\mathfrak{b}_\beta)\mathfrak{b}^\alpha\mathfrak{b}_\alpha)\\
=&\uplambda^4\upchi_1\prod_{a=1,2,\pm}((u^\alpha\Xi_\alpha)^2-\beta_a\Proj^{\alpha\beta}\Xi_\alpha\Xi_\beta)^{n_a},\label{det}\\
\end{align}
where $n_1=10$, $n_2=3$, $n_{\pm}=1$, $\beta_1=0$, and
\begin{align}
\beta_2&=\frac{\upeta}{\uplambda},\\
\beta_{\pm}&=\frac{3\uplambda  \upchi _2 c_s^2+\upchi _1 \left(4 \upeta -3\upchi _4\right)+3\upchi
   _3 \left(\uplambda +\upchi _2\right)\pm\sqrt{\Delta_D}}{6\uplambda\upchi_1}.\label{betapm}
\end{align} 
The $\Delta_D$ in \eqref{betapm} is defined in \eqref{Delta} and $\beta_\pm$ corresponds to two distinct real roots whenever \eqref{Delta} is observed.
We made successive use of the formula
\begin{align}
\det 
\begin{bmatrix}
M_1 & M_2 \\
M_3 & M_4
\end{bmatrix}
=
\det(M_1) \det(M_4 - M_3 M_1^{-1} M_2 ),
\end{align}
and also used the identity $\det(AI_4+B\mathfrak{b}^\mu\mathfrak{b}_\nu+C\mathfrak{b}^\mu\Xi_\nu)=A^3(A+(B+C)\mathfrak{b}^\mu\mathfrak{b}_\mu)$, where $\mathfrak{b}^\mu\Xi_\mu=\mathfrak{b}^\mu\mathfrak{b}_\mu=\Proj^{\alpha\beta}\Xi_\alpha\Xi_\beta$.

We now verify condition \emph{(i)}. Set $\upzeta=0$ in the above, so that
 $\Xi_\mu=\upxi_\mu$, where $\upxi^\alpha\upxi_\alpha=-(u^\alpha\upxi_\alpha)^2+\Proj^{\alpha\beta}\upxi_\alpha\upxi_\beta<0$, from \eqref{det} one obtains that $\det(A^\alpha\upxi_\alpha)\ne 0$ if $\uplambda,\upchi_1\ne 0$ and 
\begin{align}
0\le\beta_a\le 1, \, a=1,2,\pm.
\label{E:Conditions_beta_condition_i}
\end{align} 
Since $\beta_1=0$, \eqref{E:Conditions_beta_condition_i} is satisfied. As for $\beta_2$, 
\eqref{E:Conditions_beta_condition_i} is satisfied whenever \eqref{10a} is obeyed and $\uplambda,\upeta>0$. Condition \eqref{Delta} and $\upchi_1>0$ guarantee that $\beta_\pm$ are real and distinct 
with $\beta_- < \beta_+$, while \eqref{10b} sets $\beta_-> 0$ and \eqref{10b2} together with \eqref{10c} assure $\beta_+\le 1$. Then, statement \emph{(i)} in the Theorem is proved. 

The eigenvalues in \emph{(ii)} are the roots of \eqref{det} by setting $\Xi=\upzeta+\Lambda\upxi$. Reality
 of the eigenvalues $\Lambda$ are obtained by studying the roots of the polynomials $(u^\alpha\Xi_\alpha)^2-\beta\Proj^{\alpha\beta}\Xi_\alpha\Xi_\beta$ that appears in \eqref{det}. The roots of $(u^\alpha\Xi_\alpha)^2-\beta\Proj^{\alpha\beta}\Xi_\alpha\Xi_\beta=0$ are, for each one of the $\beta$'s,
\begin{align}
\Lambda_{\pm}  =(-u^\mu \upzeta_\mu u^\nu \upxi_\nu+\beta\Proj^{\mu\nu}\upxi_\mu\upzeta_\nu \pm\sqrt{\mathcal{W}})/((u^\mu \upxi_\mu)^2(1-\beta)-\beta \upxi^\mu\upxi_\mu),\label{lambda}
\end{align} 
where
\begin{align}
\mathcal{W}&=  \beta (((u^\mu\upxi_\mu)^2-\Proj^{\mu\nu}\upxi_\mu\upxi_\nu)(\Proj^{\alpha\beta}\upzeta_\alpha\upzeta_\beta-(u^\alpha\upzeta_\alpha)^2)+(u^\mu\upxi_\mu u^\nu\upzeta_\nu+\Proj^{\mu\nu}\upxi_\mu\upzeta_\nu)^2\\
&+(1-\beta)(\Proj^{\mu\nu}\upxi_\mu\upxi_\nu \Proj^{\alpha\beta}\upzeta_\alpha\upzeta_\beta-(\Proj^{\mu\nu}\upxi_\mu\upzeta_\nu)^2)).
\end{align} 
We note that these roots are always real when $0 \le \beta \le 1$ because $\Proj_{\alpha\beta}\upxi^\alpha\upxi^\beta<(\upxi_\alpha u^\alpha)^2$, $\Proj_{\alpha\beta}\upzeta^\alpha\upzeta^\beta>(\upzeta_\alpha u^\alpha)^2$, and $(\Proj^{\mu\nu}\upxi_\mu\upzeta_\nu)^2\le \Proj^{\mu\nu}\upxi_\mu\upxi_\nu \Proj^{\alpha\beta}\upzeta_\alpha\upzeta_\beta$. Then, the conditions expressed in equations \eqref{Delta}--\eqref{10c} give real eigenvalues $\Lambda$. 

We now turn to the problem of the eigenvectors for each eigenvalue. We ended up with the root $\Lambda_1$ for $\beta_1=0$ with multiplicity 20, the roots $\Lambda_{2,\pm}$ for $\beta_2=\upeta/\uplambda$ with multiplicity 3 each, and $\Lambda_{\pm,\pm}$ for $\beta_\pm$ given in \eqref{betapm} with multiplicity 1 each. The complete set of eigenvectors must contain 30 linearly independent eigenvectors. 

The roots $\Lambda_{\pm,\pm}$ are obtained from \eqref{lambda} with $\beta=\beta_\pm$ and contains 4 linearly independent eigenvectors since they are 4 distinct eigenvalues $\Lambda_{\pm,\pm}$. The remaining 26 eigenvectors appears as follows:

\vskip 0.3cm
\noindent 
$(u^\alpha\Xi_\alpha)^{20}=0$ gives
\begin{align}
\Lambda_1 &= -\frac{u^\alpha \upzeta_\alpha}{u^\beta \upxi_\beta}. 
\end{align}
There are $20$ corresponding linearly independent eigenvectors given by
\begin{align}
\begin{bmatrix}
0 \\ w^\nu_i \\ 0_{25\times 1}
\end{bmatrix},
\,
\begin{bmatrix}
 0_{25\times 1}\\1\\0_{4\times 1}
\end{bmatrix},
\,
\begin{bmatrix}
 0_{26\times 1}\\v^\nu_I
\end{bmatrix},
\,  
\begin{bmatrix}
0_{9\times 1}\\ f_0^\nu \\ f_1^\nu \\ f_2^\nu \\ f_2^\nu\\ 0_{5\times 1}
\end{bmatrix}, 
\end{align}
where $v_I^\nu$, $I=1,2,3,4$, are any $4$ linearly independent 
vectors in $\mathbb{R}^4$, $w_i^\mu=\{u^\mu,w_2^\mu,w_3^\mu\}$ are the three linearly independent vectors of $\mathbb{R}^4$ that are orthogonal to $\upzeta_\alpha+\Lambda_1\upxi_\alpha$, and 
$f_\lambda^\nu$ totalizes $16$ components that define the entries in the last vector. However, since these $16$ components are constrained by the $4$ 
equation $D^{\mss \mu\lambda}_\nu f_\lambda^\nu=0$, we end up with 12 independent entries. Then, $4+1+3+12=20$, which equals the multiplicity of the root $\Lambda_1$.

\vskip 0.3cm
\noindent 
$((u^\alpha\Xi_\alpha)-\beta_2\Proj^{\alpha\beta}\Xi_\alpha\Xi_\beta)^3=0$, where the roots $\Lambda_{2,\pm}$ are given by \eqref{lambda} with
$\beta=\beta_2 = \frac{\upeta}{\uplambda}$. Each one of these 2 roots has multiplicity 3.
The corresponding eigenvectors are
\begin{align}
\begin{bmatrix}
C_\pm\\ 
\frac{C_\pm}{\mathfrak{a}_\pm}(\mathfrak{b}_\pm)^\nu\\ 
E_\pm^\nu \\ 
\frac{(\mathfrak{b}_\pm)_0 E_\pm^\nu}{\mathfrak{a}_\pm}\\ 
\frac{(\mathfrak{b}_\pm)_{1} E_\pm^\nu}{\mathfrak{a}_\pm} \\
\frac{(\mathfrak{b}_\pm)_{2} E_\pm^\nu}{\mathfrak{a}_\pm} \\
\frac{(\mathfrak{b}_\pm)_{3} E_\pm^\nu}{\mathfrak{a}_\pm}\\ 
0_{5\times 1}
\end{bmatrix},
\end{align}
where $\mathfrak{a}_\pm=u^\alpha(\upzeta_\alpha+\Lambda_{2,\pm}\upxi_\alpha)$,
$(\mathfrak{b}_\pm)^\alpha=\Proj^{\alpha\beta}(\upzeta_\beta+\Lambda_{2,\pm}\upxi_\beta)$ (so that
$\mathfrak{a}^2_\pm=\beta_2 (\mathfrak{b}_\pm)^\mu(\mathfrak{b}_\pm)_\mu$),
\begin{align}
C_\pm &=\frac{\upeta(\Xi_\pm)_\alpha (E_\pm)^\alpha-(\upchi_4+\frac{2\upeta}{3})(\mathfrak{b}_\pm)_\alpha (E_\pm)^\alpha}{(\uplambda c_s^2+\upchi_3)\mathfrak{a}_\pm},
\end{align}
where $\Xi_\pm=\upzeta+\Lambda_{2,\pm}\upxi$,
and $E_\pm$ obeys the following constraint
\begin{align}
(\uplambda(\Xi_\pm)_\alpha+\upchi_2(\mathfrak{b}_\pm)_\alpha+\frac{\upeta}{\uplambda^2}(\uplambda^2c_s^2+\upchi_1\upeta)(\upeta(\Xi_\pm)_\alpha-(\upchi_4+\frac{2\upeta}{3})(\mathfrak{b}_\pm)_\alpha)(E_\pm)^\alpha =0.
\end{align}
Thus, the eigenvectors are written in terms of $3$ independent components of $(E_\pm)^\mu$ for each root, giving a total of $6$ eigenvectors.

\end{proof}

From the above Proposition, we immediately obtain:

\begin{corollary} \label{c:diagonalization}
Assume that $\uplambda,\upchi_1,\upeta>0$ and that \eqref{Delta}, \eqref{10a}, \eqref{10b}, and \eqref{10c} hold.
Then, the system \eqref{E:Matrix_system_1st_order} can be written as
\begin{align}
\nabla_0 \mathbf{\Psi} + \tilde{\mathcal{A}}^i \nabla_i  \mathbf{\Psi}
 & = \tilde{\mathcal{R}},
 \label{E:First_order_system_unit_coefficient}
\end{align}
where $\tilde{\mathcal{A}}^i = (\mathcal{A}^0)^{-1} \mathcal{A}^i$
and $\tilde{\mathcal{R}} = -(\mathcal{A}^0)^{-1} \mathcal{R}$,
 and the eigenvalue
problem $(\tilde{\mathcal{A}}^i  \upzeta_i - \Lambda I )V = 0$ possesses
only real eigenvalues $\Lambda$ and a set of complete eigenvectors $V$.
\end{corollary}

\section{Energy estimates}\label{Section 5}

\subsection{Preliminaries}

We begin introducing some notation.
Let $I=[0,T]$ for some $T>0$. 
We use $\mathscr{K}:\bR_+\to \bR_+$ to denote a continuous function which may vary from line to line. 
Similarly, $\mathscr{K}_I:\bR_+\to \bR_+$ denotes a continuous function depending on $I$.
Further, the notation $\fR$  always denotes a pseudodifferential operator (henceforth
abbreviated $\Psi$DO) whose mapping properties may vary from line to line.  
We denote the $L^2$ based Sobolev space of order $r$ by $H^r$, with norm $\| \cdot \|_r$.

The quasilinear nature of our equations leads us to consider a pseudodifferential
calculus for symbols with limited smoothness. Such a calculus can be found 
in \cite{Marschall-1987,Marschall-1988-Correction,Marschall-1988},
to which we will refer frequently.
For the reader's convenience, we recall the definition of these symbols and the corresponding $\Psi$DO on $\mathbb R^3$.

\begin{definition}[$\Psi$DO with limited smoothness, \cite{Marschall-1988}]
Let $k\in \bR$ and $r>3/2$. Define $\cS^k_0(r,2)(\bR^3)=\cS^k_0(r,2)(\bR^3,\bC)$ to be the space of all symbols $a:\bR^3 \times \bR^3\to \bC$ such that for all spatial multi-indices $\vec{\alpha} = (\alpha_1,
\alpha_2,\alpha_3)$
\begin{align}
\begin{split}
|\partial^{\vec{\alpha}}_\upzeta a(x,\upzeta)| & \leq C_{\vec{\alpha}}(1+|\upzeta|)^{k-|\vec{\alpha}|},\\
\|\partial^{\vec{\alpha}}_\upzeta a(x,\upzeta)\|_{H^r} & \leq C_{\vec{\alpha}}(1+|\upzeta|)^{k -|\vec{\alpha}|}.
\end{split}
\nonumber
\end{align}
For a matrix-valued symbol $a:\bR^3 \times \bR^3\to \bC^{h\times l}$ with $h,l\in \mathbb{N}$, 
we say $a \in \cS^k_0(r,2)(\bR^3,\bC^{h \times l})$ if all the entries of $a$ belong to $\cS^k_0(r,2)(\bR^3)$.
The $\Psi$DO $Op(a)$,  associated with a symbol $a\in \cS^k_0(r,2)(\bR^3,\bC^{h\times l})$ is defined by
\begin{align}
Op(a)f(x):=\frac{1}{(2\pi)^n}\int_{\bR^3} e^{\ii x\cdot \upzeta} a(x,\upzeta)\hat{f}(\upzeta)\, d\upzeta
\nonumber
\end{align}
for $f\in \cS(\bR^3,\bC^l)$, the space of Schwartz functions in $\mathbb{R}^3$, and
$\ii = \sqrt{-1}$.
\end{definition}

Having defined symbols and $\Psi$DO operators of limited smoothness in $\mathbb{R}^3$, 
we can use the coordinate invariance of the definition and standard arguments (see \cite{Marschall-1988}*{Theorem 5.1, Corollary 5.2}) to obtain a $\Psi$DO calculus on any smooth closed manifold.
In particular, we obtain such a calculus on $\mathbb{T}^3$. We denote the class of symbols on
 $\mathbb T^3$ of order $k$ with Sobolev regularity $r$ by $\cS^{k}_0(r,2)(\mathbb T^3)$
 or simply $\cS^{k}_0(r,2)$. Given
  $a \in \cS^{k}_0(r,2)(\mathbb T^3)$, we denote the $\Psi$DO associated
with $a$ by $Op(a)$ and the resulting space of $k$th order $\Psi$DO's by $OP \cS^k_0(r,2)$. 
We will not typically specify if the symbol is scalar or matrix valued since this will
be clear from the  context. 

The (flat) Laplacian on $\mathbb T^3$ is denoted by $\Delta$, and we define 
\begin{align}
\la \nabla \ra := (1 -\Delta)^{\frac 12},
\nonumber
\end{align} 
which is an element of $OP \cl S_0^1(r,2)$ for every $r \in \mathbb R$. Finally, we recall that 
\begin{align}
\| \cdot \|_{k} \simeq \| \la \nabla \ra^k \cdot \|_{0}. 
\nonumber
\end{align} 

\begin{remark}
In what follows, we will use Corollary \ref{c:diagonalization}. This Corollary follows 
from Proposition \ref{P:Diagonalization}, which involved computing
the principal symbol of \eqref{E:Matrix_system_1st_order} with
$\partial_k \mapsto \upzeta_k$. For the pseuddifferential calculus introduced above,
one uses $\partial_k \mapsto \ii \upzeta_k$ instead. In view of the homogeneity of the symbols
involved and multiplying and dividing by $\ii$ when necessary, it is not difficult
to make the two procedures compatible.
\end{remark}

\subsection{Main estimates}
We consider the linear system naturally associated with 
\eqref{E:First_order_system_unit_coefficient}. Given $\psi$, we define the
operator $\mathcal{F}(\psi)$ by
\begin{align}
\cF(\psi)\mathbf{\Psi} = \partial_t \mathbf{\Psi} + \tilde{\cA}^i (\psi) \nabla_i \mathbf{\Psi},
\nonumber
\end{align}
where $\tilde{\cA}^i(\psi)$ corresponds to the matrix
$\tilde{\cA}^i =(\cA^0)^{-1} \cA^i$ of Corollary \ref{c:diagonalization}, but with
the entries of the matrix computed using $\psi$.
Similarly, letting $\tilde{\cR}(\psi)= -(\cA^0)^{-1}(r_1,\dots,r_6)^T$
correspond to $\tilde{\cR}$ in Corollary \ref{c:diagonalization} with the entries
computed using $\psi$, we see that 
the first system~\eqref{E:Matrix_system_1st_order}, 
or, equivalently \eqref{E:First_order_system_unit_coefficient}, can be written as
\begin{equation}
\label{1st_order_sys}
\cF(\mathbf{\Psi}) \mathbf{\Psi}   = \tilde{\cR}(\mathbf{\Psi}), \quad
 \mathbf{\Psi}(0)=\mathbf{\Psi}_0,
\end{equation}

\begin{remark}
In what follows, we will often think of $\mathbf{\Psi}$ and $\psi$ as
maps from a time interval to a suitable function space.
\end{remark}

\begin{proposition}\label{p:energyest}
Let $r >  9/2$, $I \subset \bR$ and 
\begin{align}
\mathbb E_1(I) := C(I; H^r) \cap C^1(I; H^{r-1}).
\nonumber
\end{align}
There exist increasing continuous functions $\tilde M, \omega: [0,\infty) \rightarrow (0,\infty)$ such that if $\mathbf{\Psi},\psi \in C^\infty(I \times \mathbb T^3)$ satisfy  
\begin{align}
\cF(\psi) \mathbf{\Psi} & = \tilde{\cR}(\psi),  \mbox{on } I \times \mathbb T^3,
\label{1st_order_sys_eq}
\end{align}
then
\begin{align}
\norm{\mathbf{\Psi}(t)}_r^2 \leq \tilde{M}(\| \psi \|_{L^\infty(I; H^{r-1})}) e^{ 
t\omega(\|\psi \|_{\bE_1(I)})} \left [\norm{\mathbf{\Psi}_0}_r^2 + \int_0^t \norm{\tilde{\cR}(\psi(s))}_r^2\, ds \right ] \label{Basic EE}
\end{align}
 for all $t \in I$, where $\mathbf{\Psi}_0 = \mathbf{\Psi}(0)$.
\end{proposition} 

\begin{proof}
For $\upzeta = \upzeta_i dx^i \in T^* \mathbb T^3$, let  $\tilde{\cA}=\tilde{\cA}(\psi,\upzeta)=\tilde{\cA}^i(\psi) \upzeta_i$ and  $ \mathfrak{U} = Op(\tilde{\cA})$. 
From the results of Section~\ref{S:Diagonalization}, we see that 
there exist a matrix $\cS=\cS(\psi,\upzeta)$ and a diagonal matrix 
$\tilde{\mathcal{D}}=\tilde{\mathcal{D}}(\psi,\upzeta)$ such that
\begin{align}
\cS \tilde{\mathcal{A}} =  \tilde{\mathcal{D}} \cS.
\nonumber
\end{align}
Set $\mathfrak{S}:=Op(\cS)$ and $\tilde{\mathfrak{D}}:=Op(\tilde{\mathcal{D}})$.
From the expression for $\tilde{\cA}^i(\psi)\upzeta_i$, it is not difficult to see that all its entries belong to $\cS^1_0(r,2)$. Denote by $\Lambda_k=\Lambda_k(\psi,\upzeta)$ all the distinct eigenvalues 
of $\tilde{\mathcal{A}}$.
Noting that $\partial_\upzeta^{\vec{\alpha}} \tilde{\cA}(\psi,\upzeta)$ is homogeneous of degree $1-
|\vec{\alpha}| $ for $|\vec{\alpha}|\leq 1$ and $\partial_\upzeta^{\vec{\alpha}} \tilde{\cA}(\psi,\upzeta)=0$ for $|\vec{\alpha}|> 1$, we 
infer that $\Lambda_k/|\upzeta|$ is homogeneous in $\upzeta$ of degree zero.

Because the map $[(\psi,\upzeta) \mapsto \Lambda_k(\psi,\upzeta)]\in C^\infty(H^r \times T^* \mathbb T^3 , H^r)$,  
it follows that 
\begin{align}
\| \Lambda_k(\psi,\upzeta)\|_r   \leq C,\quad |\upzeta|=1
\nonumber
\end{align}
for some constant $C=C(\norm{\psi}_r)$ depending on $\norm{\psi}_r$.
By the homogeneity of $\Lambda_k/|\upzeta|$, we can conclude that
\begin{align}
\| \Lambda_k(\psi,\upzeta)\|_r   \leq C (1+|\upzeta|),
\nonumber
\end{align}
for all $\upzeta$ and some $C=C(\norm{\psi}_r)$. 
Differentiating the characteristic polynomial of $\tilde{\mathcal{A}}$ with respect to $\upzeta$ and using induction immediately yield
\begin{align}
\label{Symbol est 2}
\|\partial^{\vec{\alpha}}_\upzeta \Lambda_k(\psi,\upzeta)\|_r   \leq C_{\vec{\alpha}} (1+|\upzeta|)^{1 -
|\vec{\alpha|}},
\end{align}
for all $\upzeta$ and some $C_{\vec{\alpha}}=C_{\vec{\alpha}}(\norm{\psi}_r)$. 
This implies, by Sobolev embedding,  that $\Lambda_k\in \cS^1_0(r,2) $ and therefore
\begin{align}
\tilde{\mathfrak{D}} \in OP\cS^1_0(r,2).
\nonumber
\end{align}
The projection onto the eigenspace associated to the eigenvalue $\Lambda_k$ is given by
\begin{align}
\label{Projection}
P_k=P_k(\psi,\upzeta)= \frac{1}{2\pi \ii}\int_{\gamma_k} (z - \tilde{\cA}(\psi,\upzeta))^{-1}\, dz,
\end{align}
where $\gamma_k$ is a smooth contour enclosing only one pole $\Lambda_k$. 
By properly choosing contours $\gamma_k$, we can always make the eigenvalues $\tilde{\Lambda}_i(z,\psi,\upzeta)$ of $(z-\tilde{\cA}(\psi,\upzeta))^{-1}$ satisfy   
\begin{align}
 \| \tilde{\Lambda}_i(z,\psi,\upzeta) \|_r \leq C=C(\|\psi \|_r),\quad |\upzeta|\leq 1, \, z\in \gamma_k 
 \nonumber
\end{align}
for all $k$.
From the homogeneity of $\tilde{\cA}$ and $\Lambda_k$, we infer that $P_k$ is homogeneous of degree $0$ in $\upzeta$. 
Combining with  \eqref{Symbol est 2} and \eqref{Projection}, we obtain that 
\begin{align}
  \| P_k(\psi,\upzeta) \|_{H^r} & \leq C=C(\|\psi\|_r) ,\quad |\upzeta|=1.
  \nonumber
\end{align}
In light of the homogeneity of $P_k(\psi,\cdot)$, this implies that for all $\upzeta$ we have
\begin{align}
  \| P_k(\psi,\upzeta) \|_{H^r} & \leq C=C(\|\psi \|_r).
  \nonumber
\end{align}
For a given pair of $(\psi,\upzeta)$, we can always choose the contour $\gamma_k$ in \eqref{Projection} to be fixed in a neighborhood of  $(\psi,\upzeta)$.
Applying a similar argument to the $\upzeta$-derivatives of $P_k$ and using the homogeneity 
of $\partial_\upzeta^{\vec{\alpha}} \tilde{\cA} $, 
direct computations lead to
$P_k\in \cS^0_0(r,2)$. This implies that 
\begin{align}\label{sym est of S}
\mathcal{S}=\mathcal{S}(\psi,\upzeta) \in \cS^0_0(r,2)
\end{align}
and thus
\begin{align}
\mathfrak{S}=\mathfrak{S}(\psi) \in OP\cS^0_0(r,2)
\nonumber
\end{align}
with norm  depending on $\norm{\psi}_r$.

We can now invoke \cite{Marschall-1988}*{Corollary~3.4} to conclude
\begin{align}
\mathfrak{S} \mathfrak{U} = \tilde{\mathfrak{D} } \mathfrak{S} + \fR
\nonumber
\end{align}
with
\begin{align}
\fR \in \cL(H^s, H^s) ,\quad 1-r<s\leq r-2,
\nonumber
\end{align}
where $\cL(\mathsf{X},\mathsf{Y})$ denotes the space of linear continuous
maps between Banach spaces $\mathsf{X}$ and $\mathsf{Y}$.

We write $\mathfrak{U}=\ii \mathfrak{A} \la \nabla \ra$.
Let $\cA=\cA(\upzeta) $ denote the symbol of $\mathfrak{A}$, i.e. $\cA=-\ii \tilde{\cA}/(1+|\upzeta|^2)^{\frac 12}$. 
Therefore, $\mathfrak{A} \in OP\cS^0_0(r,2)$.
Then there exists a $\Psi$DO $\mathfrak{D} $ with symbol $\mathcal{D} \in \cS^0_0(r,2)$ such that
\begin{align}
\cS  \mathcal{A} =   \mathcal{D}  \cS  
\nonumber 
\end{align}
and thus 
\begin{align}
\mathfrak{S} \mathfrak{A} =  \mathfrak{D}  \mathfrak{S} + \fR 
\nonumber
\end{align}
with 
\begin{align}
\label{Commutator 3}
\fR \in \cL(H^{s-1}, H^s) ,\quad 1-r<s\leq r-1.
\end{align}
We can thus rewrite \eqref{1st_order_sys_eq} as 
\begin{align}
\partial_t \mathbf{\Psi} & =  \ii \mathfrak{A}(\psi) \la \nabla \ra \mathbf{\Psi} 
+\tilde{\cR}(\psi),
\nonumber
\end{align}
or 
\begin{align}
\partial_t \mathbf{\Psi} & = \mathfrak{U}(\psi) \mathbf{\Psi} +\tilde{\cR}(\psi).
\nonumber
\end{align}

Denote by $\cS^*$ the conjugate transpose matrix of $\cS$. We further set $\tilde{\mathfrak{S}}:=Op(\cS^*)$.
Note that  $  \tilde{\mathfrak{S}}=\tilde{\mathfrak{S}}(\psi)\in OP\cS^0_0(r,2)$.  
Since $\cS$ is homogeneous of degree $0$ in $\upzeta$, invoking the discussion in Section~\ref{S:Diagonalization}, we infer  that
\begin{align}
\mathbf{\Psi}^T \cS^*(\psi,\upzeta)  \cS(\psi,\upzeta) \mathbf{\Psi} \geq C_0 |\mathbf{\Psi}|^2
\nonumber
\end{align}
for some $C_0=C_0(\|\psi\|_{L^\infty}) >0$.
Let $\cB=\cB(\psi,\upzeta)=\sqrt{\cS^*(\psi,\upzeta)  \cS(\psi,\upzeta) -\frac{C_0}{2}I}$ and $\mathfrak{B}=Op(\cB)$, where $I$ is the identity matrix and, for a positive definite  matrix $A$, $B=\sqrt{A}$ denotes its square-root matrix, i.e. $B^* B=A$.
It is not difficult to see via the Cholesky algorithm that $\cB\in \cS^0_0(r,2)$ . 
Putting $\tilde{\mathfrak{B}}=Op(\cB^*)\in OP\cS^0_0(r,2)$, it follows 
from \cite{Marschall-1988}*{Corollaries~3.4 and 3.6} that
\begin{align}
\fR=&\tilde{\mathfrak{S}}\circ \mathfrak{S} -\frac{C_0}{2}I - \mathfrak{B}^* \mathfrak{B} \\
\label{S4: commutator 0}
=&[(\tilde{\mathfrak{S}}\circ \mathfrak{S} -\frac{C_0}{2}I )-  \tilde{\mathfrak{B}} \circ \mathfrak{B}]  
+(\tilde{\mathfrak{B}} \circ \mathfrak{B}  - \tilde{\mathfrak{B}}   \mathfrak{B} )  
+( \tilde{\mathfrak{B}}   \mathfrak{B} - \mathfrak{B}^* \mathfrak{B} )  \in \cL(H^{s-1},H^s)  
\end{align}
for all $1-r<s< r.$
Define
\begin{align}
N_r(t) := \la \nabla \ra^r (\frac{C_0}{2}I  + \mathfrak{B}^* \mathfrak{B})\la \nabla \ra^r .
\nonumber
\end{align}
It is an immediately conclusion from its definition that
\begin{align}
\label{S4: Garding type}
(N_r(t) \mathbf{\Psi} , \mathbf{\Psi}) \geq \frac{C_0}{2} \| \mathbf{\Psi} \|_r^2.
\end{align}
We have 
\begin{align}
\begin{split}
N_r 
=& \la \nabla \ra^r (\frac{C_0}{2}I  + \mathfrak{B}^* \mathfrak{B} - \tilde{\mathfrak{S}}\circ \mathfrak{S} )\la \nabla \ra^r  
+ \la \nabla \ra^r (\tilde{\mathfrak{S}}\circ \mathfrak{S}-\tilde{\mathfrak{S}}  \mathfrak{S}) \la \nabla \ra^r\\ 
& +  \la \nabla \ra^r  ( \tilde{\mathfrak{S}}   - \mathfrak{S}^*      ) \mathfrak{S}\la \nabla \ra^r  + \la \nabla \ra^r  \mathfrak{S}^*  \mathfrak{S} \la \nabla \ra^r.
\end{split}
\nonumber
\end{align}
It follows from  \cite{Marschall-1988}*{Corollary 3.4} that 
\begin{align}\label{S4: commutator 1}
\tilde{\mathfrak{S}}\circ \mathfrak{S}-\tilde{\mathfrak{S}}  \mathfrak{S} \in \cL(H^{s-1},H^s), \quad 1-r <s\leq r,
\end{align}
and from \cite{Marschall-1988}*{Corollary 3.6}    that 
 \begin{align}\label{S4: commutator 2}
\tilde{\mathfrak{S}}   - \mathfrak{S}^* \in \cL(H^{s-1},H^s), \quad 1-r <s< r,
\end{align}

Next, we compute
\begin{align}
\begin{split}
\frac{d}{dt} (N_r \mathbf{\Psi},  \mathbf{\Psi}) 
& = 
(N_r \frac{d}{dt} \mathbf{\Psi},  \mathbf{\Psi})
+
(N_r \mathbf{\Psi} , \frac{d}{dt} \mathbf{\Psi})
+ 
(N_r^\prime  \mathbf{\Psi},   \mathbf{\Psi})
\\
& 
= 
(N_r \mathfrak{U} \mathbf{\Psi},  \mathbf{\Psi})
+ 
(N_r  \tilde{\cR},  \mathbf{\Psi})
+
(N_r \mathbf{\Psi}, \mathfrak{U} \mathbf{\Psi})
+
(N_r \mathbf{\Psi}, \tilde{\cR})
+ 
(N_r^\prime  \mathbf{\Psi},   \mathbf{\Psi})
\\
&= 
( (N_r \mathfrak{U} + \mathfrak{U}^* N_r )\mathbf{\Psi},  \mathbf{\Psi})
+ 
(N_r  \tilde{\cR},  \mathbf{\Psi})
+
(N_r \mathbf{\Psi}, \tilde{\cR})
+ 
(N_r^\prime  \mathbf{\Psi},   \mathbf{\Psi}),
\end{split}
\nonumber
\end{align}
where ${}^\prime=\frac{d}{dt}$. 
We have
\begin{align}
\begin{split}
N_r \mathfrak{U} + \mathfrak{U}^* N_r  =& 
[\la \nabla \ra ^r (\frac{C_0}{2}I  + \mathfrak{B}^* \mathfrak{B}) \la \nabla \ra ^r  ] \ii \mathfrak{A}\la \nabla \ra  \\
& -\ii \la \nabla \ra  \mathfrak{A}^*  
[\la \nabla \ra ^r (\frac{C_0}{2}I  + \mathfrak{B}^* \mathfrak{B}) \la \nabla \ra ^r  ]
\\
 = &
\ii [\la \nabla \ra ^r (\frac{C_0}{2}I  + \mathfrak{B}^* \mathfrak{B}) \la \nabla \ra ^r\mathfrak{A}\la \nabla \ra \\
&-\la \nabla \ra  \mathfrak{A}^* \la \nabla \ra ^r (\frac{C_0}{2}I  + \mathfrak{B}^* \mathfrak{B}) \la \nabla \ra ^r ].
\end{split}
\nonumber
\end{align}

Observe that $\la \nabla \ra ^r \in OP\cS^r_0(k,2)$ for any $k$. 
We can infer from \eqref{S4: commutator 0}, \eqref{S4: commutator 1} and \eqref{S4: commutator 2} that
\begin{align}
\begin{split}
N_r \mathfrak{U} =& \ii N_r \mathfrak{A}\la \nabla \ra \\
=& \ii \la \nabla \ra ^r \mathfrak{S}^* \mathfrak{S} \la \nabla \ra ^r\mathfrak{A}\la \nabla \ra  +\fR ,
\end{split}
\nonumber
\end{align}
where $\fR=\fR(\psi) \in \cL(H^r, H^{-r})$.

To estimate the first term in the second line, we first notice that  
\cite{Marschall-1988}*{Corollary 3.4} implies
\begin{align}
\mathfrak{b}=\mathfrak{b}(\psi):=[\la \nabla \ra ^r,\mathfrak{A}(\psi) ] \in \cL(H^{r-1}, H^0),
\nonumber
\end{align}
with its norm depending on  $\|\psi\|_r$. 
Therefore, we have that 
\begin{align}
\la \nabla \ra ^r \mathfrak{S}^* \mathfrak{S} \la \nabla \ra ^r\mathfrak{A}\la \nabla \ra  & =
\la \nabla \ra ^r \mathfrak{S}^* \mathfrak{S}\mathfrak{A}\la \nabla \ra   \la \nabla \ra ^r +
\fR,
\nonumber
\end{align}
where $\fR=\fR(\psi) \in \cL(H^r, H^{-r})$.
Next, observe that by \eqref{Commutator 3}
\begin{align}
\begin{split}
\mathfrak{S}\mathfrak{A}\la \nabla \ra  & =   \mathfrak{D} \mathfrak{S}\la \nabla \ra  + \fR 
\\
& = \mathfrak{D}\la \nabla \ra  \mathfrak{S} + \fR,
\end{split}
\nonumber
\end{align}
where in   the second equality we used \cite[Corollary 3.4]{Marschall-1988}. 
We recall our convention that the operator $\fR$ may vary from line to line, both $\fR$'s 
in the last two equalities satisfy
\begin{align}
\fR=\fR(\psi) \in \cL(H^s, H^s) \quad \text{for all }-r+1<s\leq r-1.
\nonumber
\end{align} 
Therefore,
\begin{align}
N_r \mathfrak{U}  & =
\ii\la \nabla \ra ^r \mathfrak{S}^* \mathfrak{D}\la \nabla \ra  \mathfrak{S}   \la \nabla \ra ^r +
\fR,
\nonumber
\end{align}
where $\fR=\fR(\psi) \in \cL(H^r, H^{-r})$ and its norm depends on  $\|\psi\|_r$.

We can carry out a similar analysis for the term $\mathfrak{U}^* N_r$. More precisely, observe that
\begin{align}
\begin{split}
\mathfrak{U}^* N_r = & -\ii \la \nabla \ra  \mathfrak{A}^*    \la \nabla \ra ^r  \mathfrak{S}^*  \mathfrak{S} \la \nabla \ra ^r \\
&-\ii \la \nabla \ra  \mathfrak{A}^*[ \la \nabla \ra ^r (\frac{C_0}{2}I  + \mathfrak{B}^* \mathfrak{B} - \tilde{\mathfrak{S}}\circ \mathfrak{S} )\la \nabla \ra ^r \\
&  +  \la \nabla \ra ^r (\tilde{\mathfrak{S}}\circ \mathfrak{S}-\tilde{\mathfrak{S}}  \mathfrak{S}) \la \nabla \ra ^r ] \\
&  +\la \nabla \ra ^r ( \tilde{\mathfrak{S}}   - \mathfrak{S}^*      ) \mathfrak{S}\la \nabla \ra ^r.
\end{split}
\nonumber
\end{align}
Using \eqref{S4: commutator 0}, \eqref{S4: commutator 1}, \eqref{S4: commutator 2} and 
\cite{Marschall-1988}*{Theorem 2.4}, we infer that the last three terms on the right-hand side belong to $ \cL(H^r, H^{-r})$.
Since 
\begin{align}
-\ii \la \nabla \ra  \mathfrak{A}^*    \la \nabla \ra ^r  \mathfrak{S}^*  \mathfrak{S} \la \nabla \ra ^r= [\ii     \la \nabla \ra ^r  \mathfrak{S}^*  \mathfrak{S} \la \nabla \ra ^r \mathfrak{A} \la \nabla \ra   ]^*,
\nonumber
\end{align}
we  conclude that
\begin{align}
\begin{split}
N_r \mathfrak{U} + \mathfrak{U}^* N_r  =& 
\ii [\la \nabla \ra ^r \mathfrak{S}^* \mathfrak{D}\la \nabla \ra  \mathfrak{S}   \la \nabla \ra ^r\\
&-\la \nabla \ra ^r \mathfrak{S}^* \la \nabla \ra  \mathfrak{D}^* \mathfrak{S}\la \nabla \ra ^r ] + \fR_0,
\\
 = &
\ii \la \nabla \ra ^r \mathfrak{S}^*[  \mathfrak{D}\la \nabla \ra    
-  \la \nabla \ra  \mathfrak{D}^*  ]\mathfrak{S} \la \nabla \ra ^r + \fR_0,
\end{split}
\nonumber
\end{align}
where $\fR_0=\fR_0(\psi) \in \cL(H^r, H^{-r})$ with norm depending on  $\|\psi\|_r$.
The term in the parenthesis is bounded in $\cL(H^0)$ due to 
\cite{Marschall-1988}*{Corollary 3.6}.  
Therefore,
\begin{align}
\frac{d}{dt} (N_r \mathbf{\Psi},  \mathbf{\Psi}) 
 = &
\ii
(   \la \nabla \ra ^r [ \mathfrak{S}^* \mathfrak{D}\la \nabla \ra  \mathfrak{S}   
-\mathfrak{S}^* \la \nabla \ra  \mathfrak{D}^* \mathfrak{S} ] \la \nabla \ra ^r \mathbf{\Psi},  
\mathbf{\Psi})\\
& + (\fR_0 \mathbf{\Psi},  \mathbf{\Psi})
+ 
(N_r  \tilde{\cR},  \mathbf{\Psi})
+
(N_r \mathbf{\Psi}, \tilde{\cR})
+ 
(N_r^\prime  \mathbf{\Psi},   \mathbf{\Psi}). \label{eq:energy_der}
\end{align}
We also have
\begin{align}
| \ii
(   \la \nabla \ra ^r [ \mathfrak{S}^* \mathfrak{D}\la \nabla \ra  \mathfrak{S}   
-\mathfrak{S}^* \la \nabla \ra  \mathfrak{D}^* \mathfrak{S} ] \la \nabla \ra ^r \mathbf{\Psi},  \mathbf{\Psi})| 
& 
\leq C_1 \norm{\mathbf{\Psi}}_r^2,
\nonumber
\end{align}
\begin{align}
|(\fR_0  \mathbf{\Psi}, \mathbf{\Psi})|
 \leq C_2 \norm{\fR_0 \mathbf{\Psi}}_{-r}
 \norm{ \mathbf{\Psi} }_{r}
  \leq C_2
  \norm{ \mathbf{\Psi} }_{r}^2,
  \nonumber
\end{align}
\begin{align}
|(N_r  \tilde{\cR},  \mathbf{\Psi})|
+
|(N_r \mathbf{\Psi}, \tilde{\cR})|
& \leq C_3 \norm{\mathbf{\Psi}}_r \norm{\tilde{\cR}}_r \leq C_4 \norm{\mathbf{\Psi}}_r^2 + \frac{1}{2} \norm{\tilde{\cR}}_r^2.
\nonumber
\end{align}
Here the constants $C_i$ all depend on $\|\psi\|_r$.
To estimate the last term in \eqref{eq:energy_der}, observe that
\begin{align}
N^\prime(t) = \la \nabla \ra ^r \partial [\mathfrak{B}^*(\psi)  \mathfrak{B}(\psi)]  
\psi^\prime \la \nabla \ra ^r ,
\nonumber
\end{align}
where here $\partial$ stands for the Frech\'et derivative. 
From \eqref{Projection} and \eqref{sym est of S}, it is not hard to see that
\begin{align}
\partial \mathcal{B} (\psi)\psi^\prime  \in \cS^0_0(r-1,2).
\nonumber
\end{align}
Hence \cite{Marschall-1988}*{Theorem~2.3}  implies that
\begin{align}
\partial \mathfrak{B}(\psi) \psi^\prime = Op(\partial \mathcal{B} (\psi)\psi^\prime  ) \in \cL(H^0).
\nonumber
\end{align}
As $\partial \mathfrak{B}^*(\psi) \psi^\prime =[\partial \mathfrak{B}(\psi) \psi^\prime]^*$, 
we immediate conclude that
\begin{align}
\partial [\mathfrak{B}^*(\psi)  \mathfrak{B}(\psi)]  \psi^\prime \in \cL(H^0).
\nonumber
\end{align}
Now it follows that
\begin{align}
|(N_r^\prime  \mathbf{\Psi},   \mathbf{\Psi})|
& \leq C_6  \norm{N_r^\prime  \mathbf{\Psi}}_{-r} \norm{\mathbf{\Psi}}_r  \leq C_6 
\norm{ \mathbf{\Psi}}_{r}^2,
\nonumber
\end{align}
where $C_6$   depends on $\|\psi\|_{\bE_1(I)}$. In summary, 
\begin{align}
\frac{d}{dt} (N_r \mathbf{\Psi},  \mathbf{\Psi})  \leq C_7 \norm{\mathbf{\Psi}}_r^2 + C_5\norm{\tilde{\cR}}_r^2
\label{EE inequality}
\end{align}
with $C_7=C_7(\|\psi\|_{\bE_1(I)})$.
As a direct conclusion from \eqref{S4: Garding type} and Gr\"onwall's inequality, we finally conclude
\begin{align}
\norm{\mathbf{\Psi}(t)}_r^2 \leq \tilde{M} e^{ t\omega(\|\psi\|_{\bE_1(I)})} \left [\norm{\mathbf{\Psi}_0}_r^2 + \int_0^t \norm{\tilde{\cR}(\psi(s))}_r^2\, ds \right ]
\nonumber
\end{align}
where $\tilde{M}$ is a constant argument depending on $\|\psi\|_{L^\infty}$, and thus, 
on $\|\psi\|_{r-1}$ by Sobolev embedding.
\end{proof}


\section{Local existence and uniqueness}
We will now use the energy estimate of Proposition \ref{p:energyest} to establish
local well-posedness for the system \eqref{E:Matrix_system_1st_order}.

\subsection{Approximating sequence}\label{Section 6.1}
We take a sequence of smooth initial data $\mathbf{\Psi}_{0,n} \to \mathbf{\Psi}_0$ in $H^r$ with $r> 9/2$.
We inductively consider the problem
\begin{align}
\cF(\mathbf{\Psi}_{n-1}) \mathbf{\Psi}_n   &= \tilde{\cR}(\mathbf{\Psi}_{n-1}), \quad 
\mathbf{\Psi}_n(0)=\mathbf{\Psi}_{0,n}.
\label{1st_order_sys_ind}
\end{align}
Let $\norm{\mathbf{\Psi}_0}_r^2 \leq K$. We may assume
\begin{align}
\label{asp on approx initial}
\norm{\mathbf{\Psi}_{0,n}}_r^2 \leq K+1.
\end{align}
Further, we define continuous functions $\mathscr{K}_i:\bR_+ \to \bR_+$ with $i=1,2$ such that
\begin{align}
\nonumber
\| \mathfrak{A}(\psi) \|_{\cL(H^{r-1})} \leq \mathscr{K}_1(\| \psi \|_r)
\end{align}
and
\begin{align}
\nonumber
\| \tilde{\cR}(\psi) \|_{s} \leq \mathscr{K}_2(\| \psi \|_s),\quad s=r-1,r.
\end{align}
We now make the following inductive assumption
\begin{align}
\nonumber
H(n-1):  \norm{\mathbf{\Psi}_k}_{C(I;H^r)} \leq \cC_1 \, \text{ and } \, 
\|\partial_t \mathbf{\Psi}_k\|_{C(I;H^{r-1})}\leq \cC_2 \, \text{ for } \, k=1,2,\cdots, n-1.
\end{align}
Note that it follows from $H(n-1)$ and \eqref{asp on approx initial} that by choosing $T$ small enough, we have
\begin{align}
\nonumber
\norm{\mathbf{\Psi}_k(t)}_{r-1} \leq M,\quad  k=1,2,\cdots, n-1 \text{ and } t\in [0,T]
\end{align}
for some sufficiently large uniform constant $M$ independent of $\cC_i$.
Consequently, we can take the constant $\tilde{M}$ in \eqref{Basic EE} to be uniform in the 
ensuing iteration argument.

Furthermore, we choose $\cC_i$ in $H(n-1)$ sufficiently large so that
\begin{align}
\sqrt{\tilde{M} (2K+4)} \leq \cC_1
\nonumber
\end{align}
and
\begin{align}
M^\prime \mathscr{K}_1(\cC_1)\cC_1 + \mathscr{K}_2(\cC_1)\leq \cC_2,
\nonumber
\end{align}
where $M^\prime=\| \la \nabla \ra \|_{\cL(H^r,H^{r-1})}$.
Now we will use \eqref{Basic EE} to estimate 
\begin{align}
\norm{\mathbf{\Psi}_n(t)}_r^2 \leq  &     \tilde{M} e^{ t\omega( \norm{\mathbf{\Psi}_{n-1}}_{\bE_1(I)})} [\norm{\mathbf{\Psi}_{0,n}}_r^2 + \int_0^t \norm{\tilde{\cR}(\mathbf{\Psi}_{n-1}(s))}_r^2\, ds] \\
\leq & \tilde{M} e^{ t\omega( \cC_1+\cC_2)} [K+1 + t \mathscr{K}_2(\cC_1)],
\nonumber
\end{align}
By choosing $T$ small enough, we can control
\begin{align}
\nonumber
\norm{\mathbf{\Psi}_n(t)}_r^2  \leq \tilde{M}(2K+4) \quad \text{for all }t\in [0,T],
\end{align}
which gives
\begin{align}
\nonumber
\norm{\mathbf{\Psi}_n}_{C(I;H^r)} \leq \cC_1.
\end{align}
Plugging this estimate into \eqref{1st_order_sys_ind} we obtain
\begin{align}
\| \partial_t \mathbf{\Psi}_n(t)\|_{r-1} \leq M^\prime \mathscr{K}_1(\cC_1)\cC_1 + \mathscr{K}_2(\cC_1)\leq \cC_2.
\nonumber
\end{align}
This completes the verification of $H(n)$, and we conclude that 
\begin{align}
\norm{\mathbf{\Psi}_n}_{\bE_1(I)} \leq \cC
\label{Bound EE}
\end{align}
for all $n$ and some $\cC>0$.

\subsection{Energy estimate for the difference of two solutions\label{S:Difference}}
For $i=1,2$, we consider
\begin{align}
\nonumber
\cF(\psi_i) \tilde{w}_i   &= \tilde{\cR}(\psi_i), \quad \tilde{w}_i(0)=\tilde{w}_{0,i}.
\end{align}
Set $\psi=\psi_2-\psi_1$ and $\tilde{w}=\tilde{w}_2-\tilde{w}_1$. 
Take the difference of the above two systems to obtain
\begin{align}
\partial_t \tilde{w} &= \mathfrak{U}(\psi_2) \tilde{w}
 + [\mathfrak{U}(\psi_1)-\mathfrak{U}(\psi_2)] \tilde{w}_1 +\tilde{\cR}(\psi_2)-\tilde{\cR}(\psi_1)
 , \quad \tilde{w}(0)=\tilde{w}_{0,2}-\tilde{w}_{0,1}.
\label{1st_order_sys_diff}
\end{align}
Let
\begin{align}
\nonumber
\mathfrak{F}=[\mathfrak{U}(\psi_1)-\mathfrak{U}(\psi_2)] 
\tilde{w}_1 +\tilde{\cR}(\psi_2)-\tilde{\cR}(\psi_1)
\end{align}
and
\begin{align}
\nonumber
\bE_0(I):= C(I; H^{r-1})\cap C^1(I; H^{r-2}).
\end{align}
By \eqref{Basic EE}, we have
\begin{align}
\nonumber
\norm{\tilde{w}(t)}_{r-1}^2 \leq \tilde{M} e^{ t\omega} 
[\norm{\tilde{w}_{0,2}-\tilde{w}_{0,1}}_{r-1}^2 + \int_0^t \norm{\mathfrak{F}(s)}_{r-1}^2\, ds],
\end{align}
where $\tilde{M}=\tilde{M}(\norm{\psi_2}_{r-2})$ and $\omega=\omega(\norm{\psi_2}_{\bE_0(I)})$.
Estimate
\begin{align}
\begin{split}
 \|[\mathfrak{U}(\psi_1)-\mathfrak{U}(\psi_2)] \tilde{w}_1\|_{r-1}  
\leq & \int_0^1 \| \partial \mathfrak{U} (s \psi_1 + (1-s)\psi_2)(\psi) 
\tilde{w}_1\|_{r-1} ds \\
\leq & \int_0^1 \| \partial \mathfrak{A} (s \psi_1 + (1-s)\psi_2)(\psi) 
\la \nabla \ra  \tilde{w}_1\|_{r-1} ds \\
\leq & \mathscr{K}(\norm{\psi_1}_{r-1} + \norm{\psi_2}_{r-1}) \norm{\psi}_{r-1} \|\tilde{w}_1\|_r.
\end{split}
\nonumber
\end{align}
Similarly,
\begin{align}
\nonumber
\|\tilde{\cR}(\psi_2)-\tilde{\cR}(\psi_1)\|_{r-1} \leq \mathscr{K}(\norm{\psi_1}_{r-1}+ \norm{\psi_2}_{r-1}) \norm{\psi}_{r-1}.
\end{align}
This produces
\begin{align}
\norm{\tilde{w}(t)}_{r-1}^2  
\leq  & \tilde{M} e^{ t\omega } [\norm{\tilde{w}_{0,2}-\tilde{w}_{0,1}}_{r-1}^2 \\
& + t (1+\|\tilde{w}_1\|_r^2) \mathscr{K}(\norm{\psi_1}_{r-1} + \norm{\psi_2}_{r-1}) \norm{\psi}_{r-1}^2 ] .
\label{EE_diff_1}
\end{align}
Using \eqref{1st_order_sys_diff}, we further have
\begin{align}
\|\partial_t \tilde{w}\|_{r-2}  \leq \mathscr{K}(\norm{\psi_2}_{r-2})\norm{\tilde{w}}_{r-1} 
 +  \|\mathfrak{F}\|_{r-2} .
\label{EE_diff_2}
\end{align}

\subsection{Convergence of the iterates}
Now we choose $\psi_2=\tilde{w}_1=\mathbf{\Psi}_{n-1}$, $\psi_1=\mathbf{\Psi}_{n-2}$ and $\tilde{w}_2=\mathbf{\Psi}_n$. 
Note that as in Section~\ref{Section 6.1}, the constant $\tilde{M}$ in \eqref{EE_diff_1} can be taken to be independent of $n$.
Estimates \eqref{EE_diff_1} and \eqref{EE_diff_2} imply
\begin{align}
\begin{split}
&\|\mathbf{\Psi}_n-\mathbf{\Psi}_{n-1}\|_{\bE_0(I)}\\
  \leq & \sqrt{\tilde{M}} e^{\frac{T}{2} \omega(\cC)}[\| \mathbf{\Psi}_{0,n}-\mathbf{\Psi}_{0,n-1}\|_{r-1}+ \sqrt{T}(1+\cC)\mathscr{K}_I(\cC)  \|\mathbf{\Psi}_{n-1}-\mathbf{\Psi}_{n-2}\|_{\bE_0(I)} ]\\
& + \mathscr{K}_I(\cC)\sqrt{\tilde{M}} e^{\frac{T}{2} \omega(\cC)}[\| \mathbf{\Psi}_{0,n}-\mathbf{\Psi}_{0,n-1}\|_{r-1}+ \sqrt{T}(1+\cC)\mathscr{K}_I(\cC)  \| \mathbf{\Psi}_{n-1} - \mathbf{\Psi}_{n-2} \|_{\bE_0(I)} ]\\
&+ \mathscr{K}_I(\cC) \sup\limits_{t\in I} \|\mathbf{\Psi}_{n-1}(t)-\mathbf{\Psi}_{n-2}(t)\|_{r-2}.
\end{split}
\nonumber
\end{align}
In the last line, we can use \eqref{EE_diff_1}  once more to obtain
\begin{align}
\begin{split}
&\sup\limits_{t\in I} \|\mathbf{\Psi}_{n-1}(t)-\mathbf{\Psi}_{n-2}(t)\|_{r-1} \\
\leq & \sqrt{\tilde{M}} e^{\frac{T}{2} \omega(\cC)}[\|\mathbf{\Psi}_{0,n-1}-\mathbf{\Psi}_{0,n-2}\|_{r-1}+ \sqrt{T}(1+\cC)\mathscr{K}_I(\cC)  \|\mathbf{\Psi}_{n-2}-\mathbf{\Psi}_{n-3}\|_{\bE_0(I)} ].
\end{split}
\nonumber
\end{align}
We can choose $T$ small and $(\mathbf{\Psi}_{0,n})$ in such a way that
\begin{align}
\nonumber
(1+\mathscr{K}_I(\cC))\sqrt{\tilde{M}} e^{\frac{T}{2} \omega(\cC)}(\|\mathbf{\Psi}_{n-1}(t)-\mathbf{\Psi}_{n-2}(t)\|_{r-1} + \| \mathbf{\Psi}_{0,n-1}-\mathbf{\Psi}_{0,n-2}\|_{r-1} ) \leq 2^{-n},
\end{align}
\begin{align}
\nonumber
\sqrt{\tilde{M}} e^{\frac{T}{2} \omega(\cC)}\sqrt{T}(1+\cC)\mathscr{K}_I^2(\cC)  \leq 1/16,
\end{align}
and
\begin{align}
\nonumber
\sqrt{\tilde{M}} e^{\frac{T}{2} \omega(\cC)}\sqrt{T}(1+\cC)(\mathscr{K}_I^2(\cC) +\mathscr{K}_I(\cC) )  \leq 1/4.
\end{align}
Set $a_n=\|\mathbf{\Psi}_n-\mathbf{\Psi}_{n-1}\|_{\bE_0(I)}$. We then have
\begin{align}
\nonumber
a_n \leq 2^{-n} + a_{n-1}/4 + a_{n-2}/16.
\end{align}
Using induction, it follows that
\begin{align}
a_n \leq \frac{s_n}{2^{2n-3}} + \frac{F_n}{2^{2n-4}}a_2 + \frac{F_{n-1}}{2^{2n-2}} a_1,
\label{Inductive: a_n}
\end{align} 
where $F_n$ is the $n-$th term of Fibonacci sequence (starting from $0$) and
\begin{align}
\nonumber
s_n=2^{n-3} + s_{n-1}+s_{n-2}.
\end{align}
Letting $b=(1-\sqrt{5})/2$, we obtain the following identities 
\begin{align}
s_n - b s_{n-1} &= 2^{n-3} + (1-b)(s_{n-1} - b s_{n-2})
\nonumber
\\
b(s_{n-1} - b s_{n-2}) &= 2^{n-4} b  + b(1-b)(s_{n-2} - b s_{n-3})
\nonumber
\\
& \vdots\\
b^{n-3}(s_3 - b s_2) &=  b^{n-3}  + b^{n-3}(1-b)(s_2 - b s_1).
\nonumber
\end{align}
We sum these expressions to conclude
\begin{align}
\nonumber
s_n - b^{n-2} s_2 = \sum_{k=0}^{n-3} 2^k b^{n-3-k} + (1-b)(s_{n-1}  -b^{n-2} s_1).
\end{align}
We carry out a similar computation and sum to obtain
\begin{align}
s_n - (1-b) s_{n-1} =& \sum_{k=0}^{n-3} 2^k b^{n-3-k} + b^{n-2} s_2  - (1-b)b^{n-2} s_1 
\nonumber
\\
(1-b)( s_{n-1} - (1-b) s_{n-2} )  =& (1-b)\sum_{k=0}^{n-4} 2^k b^{n-4-k} + (1-b)b^{n-3} s_2  - (1-b)^2 b^{n-3} s_1 
\nonumber
\\
& \vdots
\nonumber
\\
(1-b)^{n-3}( s_3 - (1-b) s_2 )  =& (1-b)^{n-3}  + (1-b)^{n-3}b s_2  - (1-b)^{n-2} b s_1 .
\nonumber
\end{align}
This yields
\begin{align}
s_n - (1-b)^{n-2} s_2 =&[\sum_{k=0}^{n-3} 2^k b^{n-3-k} + (1-b)\sum_{k=0}^{n-4} 2^k b^{n-4-k} + \cdots + (1-b)^{n-3}] 
\nonumber
\\
&+ s_2\sum_{k=1}^{n-2}b^k (1-b)^{n-2-k} + s_1\sum_{k=1}^{n-2}b^{n-1-k} (1-b)^k
\nonumber
\end{align}
and thus
\begin{align}
s_n \leq (n-2) 2^{n-3} + F_n s_2   \leq  (n-2) 2^{n-3} + 2^{n-3}  s_2,
\nonumber
\end{align}
where we have used that $(1-b)^k=F_{k+2} -F_{k+1} b$.
Plugging this expression into \eqref{Inductive: a_n} gives 
\begin{align}
\nonumber
a_n \leq \frac{n-2}{2^n} + \frac{s_2}{2^n}  + \frac{a_2}{2^{n-4}} +\frac{a_1}{2^{n-2}}.
\end{align}
Therefore
\begin{align}
\nonumber
\|\mathbf{\Psi}_n-\mathbf{\Psi}_{n+j}\|_{\bE_0(I)} \leq \|\mathbf{\Psi}_n-\mathbf{\Psi}_{n+1}\|_{\bE_0(I)}+\cdots \|\mathbf{\Psi}_{n+j-1}-\mathbf{\Psi}_{n+j}\|_{\bE_0(I)}
\end{align}
can be made arbitrarily small by taking $n$ large. We thus conclude 
that $\{\mathbf{\Psi}_n\}$ is Cauchy in $C(I; H^{r-1})\cap C^1(I; H^{r-2})$  and therefore converges in this space. 
We denote the limit by $\mathbf{\Psi}\in C(I; H^{r-1})\cap C^1(I; H^{r-2})$. We can let 
$n\to \infty$ in \eqref{1st_order_sys_ind} and thus $\mathbf{\Psi}$ satisfies
\begin{align}
\cF(\mathbf{\Psi}) \mathbf{\Psi}   &= \tilde{\cR}(\mathbf{\Psi}), \quad \mathbf{\Psi}(0)=\mathbf{\Psi}_0.
\nonumber
\end{align}
Finally, it follows from \eqref{Bound EE} that
\begin{align}
\norm{\mathbf{\Psi}(t)}_r + \| \partial_t \mathbf{\Psi}(t)\|_{r-1} \leq \cC,\quad t\in [0,T].
\nonumber
\end{align}

It remains to prove uniqueness. But this follows at once since we have an estimate for the 
difference of two solutions.

\subsection{Continuity of solution}
The proof of time-continuity of the solution with respect to the top norm follows
a standard procedure: first, we prove weak continuity; then, we obtain strong
continuity by showing continuity of the norm.

The weak continuity of the solution $\mathbf{\Psi}$ can be proved by a similar argument to that of 
quasilinear wave equations, since in that proof the structure of the equation is not necessary but only the convergence $\mathbf{\Psi}_n\to \mathbf{\Psi} $ in $C(I; H^{r-1})\cap C^1(I; H^{r-2})$ and an estimate of the form~\eqref{Bound EE} are used.

We put
\begin{align}
\nonumber
\mathfrak{K}(t)=\sqrt{\frac{C_0}{2}I + (\mathfrak{B}(\mathbf{\Psi}(t)))^* \mathfrak{B}(\mathbf{\Psi}(t)) }
\end{align}
and
\begin{align}
\nonumber
\cA_r(t)=\cA_r(\mathbf{\Psi}(t))=  \mathfrak{K}(t) \la \nabla \ra ^r.
\end{align}
Hence
\begin{align}
\nonumber
N_r(t)=N_r(\mathbf{\Psi}(t)) = \cA_r(\mathbf{\Psi}(t))^*  \cA_r(\mathbf{\Psi}(t)).
\end{align}
Recall that $\mathfrak{B}\in OP \cS^0_0(r,2)$. 
From \cite{Marschall-1988}*{Theorems~2.2 and 2.4}, it follows that
\begin{align}
\label{S5: mapping K}
\mathfrak{K}(t) \in \cL(H^s),\quad -r<s<r-1.
\end{align}
Fix $t_0\in [0,T]$. We will show that $\cA_r(t_0) \mathbf{\Psi}(t)$ is weakly continuous in $H^0$. Given any $\epsilon>0$ and $\phi\in H^0$, take a sequence of Schwartz's functions $\phi_j\to \phi $ in $H^0$. Then
\begin{align}
&(\cA_r(t_0) \mathbf{\Psi}(t)- \cA_r(t_0) \mathbf{\Psi}_n(t), \phi) 
\nonumber
\\
=& (\cA_r(t_0) \mathbf{\Psi}(t)- \cA_r(t_0) \mathbf{\Psi}_n(t), \phi - \phi_j) 
+ (\cA_r(t_0) \mathbf{\Psi}(t)- \cA_r(t_0) \mathbf{\Psi}_n(t), \phi_j)
\nonumber
\end{align}
In view of \eqref{Bound EE}, the first term on the RHS is bounded by
\begin{align}
\nonumber
|(\cA_r(t_0) \mathbf{\Psi}(t)- \cA_r(t_0) \mathbf{\Psi}_n(t), \phi - \phi_j)| \leq \mathscr{K}(\cC) \|\phi - \phi_j\|_0,
\end{align}
so that this term can be made less than $\epsilon/2$ upon choosing $j$ large enough.
Next, fixing $j$ in the second term, we have
\begin{align}
\begin{split}
& |(\cA_r(t_0) \mathbf{\Psi}(t)- \cA_r(t_0) \mathbf{\Psi}_n(t), \phi_j)| \\
= & |( \la \nabla \ra ^{r-1} ( \mathbf{\Psi}(t)-   \mathbf{\Psi}_n(t)),  
\la \nabla \ra   \mathfrak{K}^*(t) \phi_j)|
\end{split}
\nonumber
\end{align}
Because $\mathbf{\Psi}_n\to \mathbf{\Psi} $ in $C(I; H^{r-1})$,
invoking \cite{Marschall-1988}*{Theorem~2.4} and \eqref{S5: mapping K}, we obtain
\begin{align}
\nonumber
|(\cA_r(t_0) \mathbf{\Psi}(t)- \cA_r(t_0) \mathbf{\Psi}_n(t), \phi_j)| <\epsilon/2
\end{align}
for all $n \geq n_0$ with some large enough $n_0$. Hence,
\begin{align}
\nonumber
|(\cA_r(t_0) \mathbf{\Psi}(t)- \cA_r(t_0) \mathbf{\Psi}_n(t), \phi) | <\epsilon \quad \text{for all }n\geq n_0 \text{ and } t\in [0,T],
\end{align}
showing that $\cA_r(t_0) \mathbf{\Psi}_n(t) $ converges to $\cA_r(t_0) \mathbf{\Psi}(t)$ uniformly in $t$ in the weak topology. Thus,
$\cA_r(t_0) \mathbf{\Psi}(t)$ is weakly continuous in $t$ with respect to the norm of $H^0$.

We are now ready to show that $ \mathbf{\Psi}  \in C(I; H^r)$. In view of the weak continuity of 
$ \mathbf{\Psi}(t)$, it suffices to demonstrate that the map 
\begin{align}
\nonumber
[t\mapsto \|   \mathbf{\Psi} (t) \|_r] \quad \text{is continuous.}
\end{align}
Applying \eqref{EE inequality} to \eqref{1st_order_sys} and using\eqref{Bound EE}, we conclude that
\begin{align}
\frac{d}{dt}\| \cA_r(t) \mathbf{\Psi}(t)\|^2_0  \leq \mathscr{K}(\cC).
\label{Basic EE 2}
\end{align}
This implies that
\begin{align}\label{Lip cont}
 \| \cA_r(t) \mathbf{\Psi}(t)\|^2_0 =:Y(t) \quad \text{is Lipschitz continuous in } t.
\end{align}
Consider
\begin{align}
& \| \cA_r(t_0) \mathbf{\Psi}(t)\|^2_0 - \| \cA_r(t_0) \mathbf{\Psi}(t_0)\|^2_0  
\nonumber
\\
=&   (\| \cA_r(t_0) \mathbf{\Psi}(t)\|^2_0 - \| \cA_r(t) \mathbf{\Psi}(t)\|^2_0   ) 
+  (\| \cA_r(t) \mathbf{\Psi}(t)\|^2_0 - \| \cA_r(t_0) \mathbf{\Psi}(t_0)\|^2_0   ).
\nonumber
\end{align}
The first term on the RHS can be estimated as follows.
\begin{align}
& |(\| \cA_r(t_0) \mathbf{\Psi}(t)\|^2_0 - \| \cA_r(t) \mathbf{\Psi}(t)\|^2_0   )| 
\nonumber
\\
= &|\| \cA_r(t_0) \mathbf{\Psi}(t)\|_0 - \| \cA_r(t) \mathbf{\Psi}(t)\|_0   |(\| \cA_r(t_0) \mathbf{\Psi}(t)\|_0 + \| \cA_r(t) \mathbf{\Psi}(t)\|_0   )
\nonumber
\\
\leq & \mathscr{K}(\cC) \| (\mathfrak{K}(t)-\mathfrak{K}(t_0)) \la \nabla \ra ^r   \mathbf{\Psi}(t)\|_0  
\nonumber
 \\
\leq & \mathscr{K}(\cC) \| \mathfrak{K}(t)-\mathfrak{K}(t_0)\|_{\cL(H^0)}.
\nonumber
\end{align}
As elements in $\cL(H^0)$,  it is not hard to check that $\mathfrak{K}(\mathbf{\Psi})$ depends continuously on $\norm{\mathbf{\Psi}}_{r-1}$.
Combining with \eqref{Lip cont}, this observation shows 
that $[t \mapsto \|\cA_r(t_0) \mathbf{\Psi} (t)\|_0] $ is continuous at $t_0$; and thus 
\begin{center}
$\cA_r(t_0) \mathbf{\Psi} (t)$ is continuous in $t$ at $t_0$ w.r.t. $H^0$. 
\end{center}
Since $t_0$ is arbitrary, from
\begin{align}
\|\mathbf{\Psi}(t )-\mathbf{\Psi}(t_0)\|_r^2 \leq C \|\cA_r(t_0) \mathbf{\Psi} (t ) 
- \cA_r(t_0) \mathbf{\Psi} (t_0)\|_0^2,
\nonumber
\end{align}
we obtain that $\mathbf{\Psi}   \in C(I; H^r)$.
Using this fact and equation~\eqref{1st_order_sys_ind}, we immediately conclude 
\begin{align}
\nonumber
\mathbf{\Psi}   \in C(I; H^r) \cap C^1(I, H^{r-1}).
\end{align}

\section{Solution to the original system} 
It remains to prove that the solution $\mathbf{\Psi}$ that we obtained for the system
\eqref{E:Matrix_system_1st_order} yields a solution to \eqref{E:u_unit}
and \eqref{E:Div_T}. The argument follows a known approximation argument by 
analytic functions, so the main task is to show that the analytic
Cauchy problem for \eqref{E:u_unit} and \eqref{E:Div_T} can be solved.
In fact, due to the compactness of $\mathbb{T}^3$ and some minor technical points
related to a localization procedure, we work in a slightly larger space than that 
of analytic functions, namely, Gevrey spaces.

\begin{remark}
\label{R:Gevrey_original}
In the introduction, we alluded to a local well-posendess in Gevrey spaces for
equations \eqref{E:u_unit} and \eqref{E:Div_T}. However, what was actually
established in \cite{DisconziBemficaNoronhaBarotropic} was the local
well-posedness in Gevrey spaces for equations \eqref{E:EofM}, i.e., the projection
onto the spaces orthogonal and parallel to $u$ of equations \eqref{E:u_unit} and \eqref{E:Div_T}. 
That this solutions in fact produces a solution
to \eqref{E:u_unit} and \eqref{E:Div_T} was 
not showed in \cite{DisconziBemficaNoronhaBarotropic}. Because we do need
Gevrey solutions to \eqref{E:u_unit} and \eqref{E:Div_T}
to carry out the aforementioned approximation argument, we cannot rely on 
\cite{DisconziBemficaNoronhaBarotropic} and, therefore, we establish the
desired result here.
\end{remark}

\begin{remark}
In practice (e.g., when implementing numerical simulations), physicists do not use equations
\eqref{E:u_unit} and \eqref{E:Div_T}, adopting instead \eqref{E:EofM} as the starting point.
Therefore, for applications in physics, the result obtained in \cite{DisconziBemficaNoronhaBarotropic}
(see previous Remark) is enough as far as only Gevrey regularity is concerned.
\end{remark}

\subsection{The original equations in explicit form}
Equation \eqref{E:Div_T} reads
\begin{align}
& (\varepsilon+A_1+A_2+P)(u^\mu\nabla_\nu u^\nu+u^\nu\nabla_\nu u^\mu)+u^\mu u^\nu\nabla_\nu(\varepsilon+A_1)+\Proj^{\mu\nu}\nabla_\nu(P+A_2)\\
&-2\upsigma^{\mu\nu}\nabla_\nu\upeta-2\upeta\nabla_\nu\upsigma^{\mu\nu}+Q^\mu\nabla_\nu u^\nu+u^\nu\nabla_\nu Q^\mu+Q^\nu\nabla_\nu u^\mu+u^\mu\nabla_\nu Q^\nu=0.\label{E:DivT2}
\end{align} 
Applying $u^\mu \nabla_\mu$ twice to \eqref{E:u_unit} produces:
\begin{align}
&u_\nu u^\alpha u^\beta\nabla_\alpha\nabla_\beta u^\nu+u^\alpha u^\beta(\nabla_\alpha u_\nu)(\nabla_\beta u^\nu)=0.\label{E:Constraint} 
\end{align}
We may rewrite the above equations to obtain the complete set of equations given by
\begin{subequations}{\label{Eq4}}
\begin{align}
& u_\nu u^\alpha u^\beta\nabla_\alpha\nabla_\beta u^\nu=B(\partial u),\label{4a}\\
& B^{\mu\alpha\beta}
\partial_\alpha \partial_\beta \varepsilon +B^{\mu\alpha\beta}_\nu\partial_\alpha
\partial_\beta u^\nu=B^\mu(\partial \varepsilon,\partial u,\partial g),\label{4b}
\end{align}
\end{subequations}
where
\begin{align}
&B^{\mu\alpha\beta}=\frac{\upchi_1 u^\mu u^\alpha u^\beta+(\upchi_3+\uplambda c_s^2) u^{(\alpha}\Proj^{\beta)\mu}+\uplambda c_s^2 u^\mu\Proj^{\alpha\beta} }{(\varepsilon+P)}\\
&B^{\mu\alpha\beta}_\nu=(\upchi_2+\uplambda)u^\mu u^{(\alpha}\delta^{\beta)}_\nu+(\upchi_4-\frac{\upeta}{3})\Proj^{\mu(\alpha}\delta^{\beta)}_\nu-\upeta\Proj^{\alpha\beta}\delta^\mu_\nu+\uplambda u^\alpha u^\beta\delta^\mu_\nu,
\end{align} 
$B(\partial u)$ and $B^\mu(\partial \varepsilon,\partial u)$
are analytic functions of $\varepsilon$, $u$, and their first order derivatives and
contain no term in second order derivatives of $\varepsilon$ or $u$.
By constructing the vector $U=(\varepsilon,u^\alpha) \in\mathbb{R}^{5}$, 
we may write the above equations in matrix form as 
$\mathfrak{m}^{\alpha\beta}\partial^2_{\alpha\beta}U=\mathcal{B}$, where $\mathcal{B}=(B,B_\mu)\in\mathbb{R}^{5}$ and
\begin{align}
\mathfrak{m}^{\alpha\beta}=\begin{bmatrix}
0 & u^\alpha u^\beta u_\nu\\
B^{\mu\alpha\beta} & B^{\mu\alpha\beta}_\nu
\end{bmatrix}.
\end{align}

\subsection{The characteristic determinant\label{S:Characteristics}}
We will here compute the characteristic determinant of equations 
\eqref{Eq4}. Let $\upxi$ be an arbitrary co-vector in spacetime. 
To simplify the notation, denote $\mathfrak{a}:=u^\mu\upxi_\mu$ and $\mathfrak{b}^\mu:=\Proj^{\mu\nu}\upxi_\nu$, and 
\begin{align}
B^\mu& :=B^{\mu\alpha\beta}\upxi_\alpha\upxi_\beta=
\frac{\upchi_1 u^\mu \mathfrak{a}^2+(\upchi_3+\uplambda c_s^2) \mathfrak{a}\mathfrak{b}^\mu+\uplambda c_s^2 u^\mu \mathfrak{b}^\alpha\mathfrak{b}_\alpha }{(\varepsilon+P)},\\
B^\mu_\nu& :=B^{\mu\alpha\beta}_\nu\upxi_\alpha\upxi_\beta
=(\mathfrak{a}(\upchi_2+\uplambda)u^\mu +(\upchi_4-\frac{\upeta}{3})\mathfrak{b}^\mu)\upxi_\nu+(\uplambda \mathfrak{a}^2-\upeta \mathfrak{b}^\alpha\mathfrak{b}_\alpha)\delta^\mu_\nu.
\end{align}
Also, let us define 
\[C\equiv[C^A_B]_{5\times 5}:=\mathfrak{m}^{\alpha\beta}\upxi_\alpha \upxi_\beta=\begin{bmatrix}
0 & \mathfrak{a}^2 u_\nu\\
B^{\mu} & B^{\mu}_\nu
\end{bmatrix},\] 
where $A,B=\bar{1},\mu=\bar{1},0,1,2,3$, with $\bar{1}$ denoting the first line or column index. Then, $C^{\bar{1}}_{\bar{1}}=0$, $C^{\bar{1}}_\nu=\mathfrak{a}^2 u_\nu$, $C^\mu _{\bar{1}}=B^\mu$, and $C^\mu_\nu=B^\mu_\nu$. Using the Levi-Civita symbol $\epsilon^{ABCDE}$ with $\epsilon^{\bar{1}0123}=\epsilon_{\bar{1}0123}=1$, where $\epsilon^{ABCDE}\epsilon_{FGHIJ}=5!\updelta^A_{[F}\updelta^B_G\updelta^C_H\updelta^D_I\updelta^E_{J]}$ with the bracket ${}_{[\cdots]}$ being the anti-symmetrization of the indexes $FGHI$, we obtain that
\begin{align}
&\det(\mathfrak{m}^{\alpha\beta}\upxi_\alpha\upxi_\beta)=\det(C)=\frac{\epsilon_{A_1A_2A_3A_4A_5}\epsilon^{B_1B_2B_3B_4B_5}}{5!}C^{A_1}_{B_1}
C^{A_2}_{B_2}C^{A_3}_{B_3}C^{A_4}_{B_4}C^{A_5}_{B_5}\\
&=C^{\bar{1}}_{\nu}\frac{\epsilon^{\nu B_2B_3B_4B_5}\epsilon_{\bar{1}\mu_2\mu_3\mu_4\mu_5}}{4!}
C^{\mu_2}_{B_2}C^{\mu_3}_{B_3}C^{\mu_4}_{B_4}C^{\mu_5}_{B_5}
=-C^{\bar{1}}_{\nu}C^\mu_{\bar{1}}\frac{\epsilon^{\bar{1}\nu\nu_3\nu_4\nu_5}\epsilon_{\bar{1}\mu\mu_3\mu_4\mu_5}}{3!}
C^{\mu_3}_{\nu_3}C^{\mu_4}_{\nu_4}C^{\mu_5}_{\nu_5}\\
&=-\mathfrak{a}^2u_\nu B^\mu\frac{\epsilon^{\bar{1}\nu\nu_2\nu_3\nu_4}\epsilon_{\bar{1}\mu\mu_2\mu_3\mu_4}}{3!}B^{\mu_2}_{\nu_2}B^{\mu_3}_{\nu_3}B^{\mu_4}_{\nu_4}\\
&=-\mathfrak{a}^2u_\nu B^\mu(\uplambda \mathfrak{a}^2-\upeta \mathfrak{b}^\alpha\mathfrak{b}_\alpha)^3\frac{\epsilon^{\bar{1}\nu\mu_2\mu_3\mu_4}\epsilon_{\bar{1}\mu\mu_2\mu_3\mu_4}}{3!}\\
&-\mathfrak{a}^2u_\nu B^\mu(\mathfrak{a}(\upchi_2+\uplambda)u^{\mu_2} +(\upchi_4-\frac{\upeta}{3})\mathfrak{b}^{\mu_2})\upxi_{\nu_2}(\uplambda \mathfrak{a}^2-\upeta \mathfrak{b}^\alpha\mathfrak{b}_\alpha)^2\frac{\epsilon^{\bar{1}\nu\nu_2\mu_3\mu_4}\epsilon_{\bar{1}\mu\mu_2\mu_3\mu_4}}{2!}\\
&=\frac{\mathfrak{a}^2}{\varepsilon+P}(\uplambda \mathfrak{a}^2-\upeta \mathfrak{b}^\alpha\mathfrak{b}_\alpha)^2(\lambda (\upchi_1 \mathfrak{a}^2-\upchi_3\mathfrak{b}^\alpha\mathfrak{b}_\alpha)\mathfrak{a}^2-\upchi_2(\upchi_3+\uplambda c_s^2 ) \mathfrak{a}^2\mathfrak{b}^\alpha\mathfrak{b}_\alpha\\
&+(\upchi_1 \mathfrak{a}^2+\uplambda c_s^2)(\upchi_4-\frac{4\upeta}{3}\mathfrak{b}^\beta\mathfrak{b}_\beta)\mathfrak{b}^\alpha\mathfrak{b}_\alpha)\\
&=\frac{\uplambda^3\upchi_1}{\varepsilon+P}\prod_{a=1,2,\pm}((u^\alpha\upxi_\alpha)^2-\beta_a \Proj^{\alpha\beta}\upxi_\alpha\upxi_\beta)^{m_a},
\end{align}
where $m_1=1$, $m_2=2$, and $m_\pm=1$, while the $\beta_a$'s are the same as the one obtained in \eqref{det}. 

We have already showed that conditions $\uplambda,\upeta,\upchi_1>0$ together with \eqref{Delta}--\eqref{10c} (which are the assumptions in Theorem \ref{T:main_theorem}) guarantee that $0\le \beta_a\le 1$. 
Therefore, comparing $(u^\alpha\upxi_\alpha)^2-\beta_a \Proj^{\alpha\beta}\upxi_\alpha\upxi_\beta$
with the characteristics of an acoustical metric (see, e.g., \cite{DisconziSpeckRelEulerNull}),
we conclude that $\det(\mathfrak{m}^{\alpha\beta}\upxi_\alpha\upxi_\beta)$ is a product
of hyperbolic polynomials.

\begin{remark}[The system's characteristics]
Setting $\det(\mathfrak{m}^{\alpha\beta}\upxi_\alpha\upxi_\beta)$ equal to zero, we obtain
the characteristics of the system \eqref{Eq4}. Not surprisingly, these are the same as the characteristics
of equations \eqref{E:EofM} which have been computed in \cite{DisconziBemficaNoronhaBarotropic}.
\end{remark}

\subsection{Proof of Theorem \ref{T:main_theorem}}.
Let us group that unknowns $\varepsilon$, $u^\alpha$ in the
$5$-component vector $V=(\varepsilon,u^\alpha)$. 
To each component $V^I$ we associate an index $m_I$, $I=1,\dots,5$, and to each one of the $5$ equations \eqref{4a}-\eqref{4b} we associate an index $n_J$, in such a way that equations
\eqref{4a}-\eqref{4b} can be written as
\begin{gather}
 h^J_I(\partial^{m_K-n_J-1}V^K,\partial^{m_I-n_J})V^I+b^J(\partial^{m_K-n_J-1}V^K)=0,
\label{system_general}
\end{gather}
where $I,J=1,\dots,5$, $ h^J_I(\partial^{m_K-n_J-1}V^K,\partial^{m_I-n_J})$
is a homogeneous differential operator of order $m_I - n_J$ (which could possibly be zero)
whose coefficients depend on at most $m_K - n_J - 1$ derivatives of $V^K$, 
$K=1,\dots, 5$, and there is a sum over $I$ in  
$h^J_I(\cdot)V^I$. The terms $b^J(\partial^{m_K-n_J-1}V^K)$ also depend on at most
$m_K - n_J - 1$ derivatives of $V^K$, $K=1,\dots, 5$. The indices $m_I$ and $n_J$ are
defined up to an overall additive constant, but the simplest choice to have equations
\eqref{4a}, \eqref{4b} written as \eqref{system_general} is $m_I = 2$, $n_J = 0$, for all 
$I,J=1,\dots,5$.

Using the fact that the characteristic determinant of \eqref{system_general} computed above 
is a product of hyperbolic polynomials, we conclude that \eqref{system_general}
forms a Leray-Ohya system (see \cite{DisconziExistenceCausalityConformal}). We can then apply
theorems A.18 and A.23 of \cite{DisconziExistenceCausalityConformal}
(whose proofs can be found in  \cite{Choquet-Bruhat-1966,LerayOhyaHyperbolicNonStrictNonLinear-1967}) to conclude that equations
\eqref{Eq4} are locally well-posed in suitable\footnote{From the previously mentioned
theorems, it is not difficult to see that one can be very precise about the quantitative
properties of solutions, including the exact Gevrey regularity. Such details, however, are not
important here for our argument.} Gevrey spaces.

\begin{remark}
Theorems A.18 and A.23 are applicable to systems in $[0,T]\times \mathbb{R}^n$, whereas
here we have $[0,T]\times\mathbb{T}^3$. Thus, one needs to carry out a localization
and gluing argument before invoking these theorems. 
Such argument is possible due to the existence of a domain of dependence for solutions 
guaranteed by Theorem A.19 of \cite{DisconziExistenceCausalityConformal}.
The procedure is exactly the same
as in \cite{DisconziExistenceCausalityConformal} so we will not present it here.
\end{remark}

Equation \eqref{4a} can be written as
\begin{align}
u^\mu u^\nu \nabla_\mu \nabla_\nu (u_\alpha u^\alpha) = 0.
\nonumber
\end{align}
Viewing this as an equation for $u^\alpha u_\alpha$, we see that it forms a Leray-Ohya equation,
so it admits unique solutions in Gevrey spaces. Therefore, we conclude that a solution
to \eqref{system_general} satisfies $u^\alpha u_\alpha = -1$ provided that this condition holds
initially, which is the by construction (see comment after the statement of Theorem \ref{T:main_theorem}).

The conclusion that $\mathbf{\Psi}$ yields a solution to \eqref{E:u_unit} and \eqref{E:Div_T} now follows
from a known approximation argument, so we will be brief. 

Consider the initial data $\mathcal{I} = (\varepsilon_{(0)},\varepsilon_{(1)},
u_{(0)}, u_{(1)}) \in H^r$ for \eqref{E:u_unit}-\eqref{E:Div_T}
and let $\mathcal{I}_k$ be a sequence of Gevrey regular
data converging to $\mathcal{I}$ in $H^r$. For each $k$, 
let $V_k=(\varepsilon_k,u_k)$ be the Gevrey regular 
solution to \eqref{E:u_unit}-\eqref{E:Div_T} with data $\mathcal{I}_k$, 
whose existence is ensured by the foregoing discussion.
In view of the way \eqref{E:Matrix_system_1st_order}
was derived from \eqref{E:u_unit}-\eqref{E:Div_T}, for each $k$, we obtain a 
Gevrey regular solution $\mathbf{\Psi}_k$ to \eqref{E:Matrix_system_1st_order}, 
with $\mathbf{\Psi}_k$ defined in terms of
$V_k$ according to the definitions of Section \ref{S:New_system}.

Let $\mathbf{\Psi}_0$ be initial data for \eqref{E:Matrix_system_1st_order} constructed
out of $\mathcal{I}$, i.e., we define $\mathbf{\Psi}_0$ in terms of 
$\mathcal{I}$ using the definitions of Section \ref{S:New_system}. 
This is possible since the entries of $\mathbf{\Psi}_0$ 
will be simple algebraic expressions in terms of $\mathcal{I}$.

Let $\mathbf{\Psi}$ be the solution to \eqref{E:Matrix_system_1st_order}
with data $\mathbf{\Psi}_0$. Note that 
we do not assume that $\mathbf{\Psi}$ is given in terms 
of the original fluid variables via the relations of Section \ref{S:New_system}
since at this point we do not yet have a solution to 
\eqref{E:u_unit}-\eqref{E:Div_T} with data $\mathcal{I}$. 
In other words, the entries of $\mathbf{\Psi}$ are 
treated as independent variables; at this point
the only relation between $\mathbf{\Psi}$ and the original system \eqref{E:u_unit}-\eqref{E:Div_T} 
is that $\mathbf{\Psi}_0$ is constructed out of $\mathcal{I}$. 

The estimates for solutions to \eqref{E:Matrix_system_1st_order}
derived in Section \ref{Section 5} combined with the estimates for the difference of solutions
in Section \ref{S:Difference},
imply that as $\mathcal{I}_k \rightarrow
\mathcal{I}$ in $H^r$, $\mathbf{\Psi}_k$ converges
to $\mathbf{\Psi}$, and thus the solutions $V_k$ to \eqref{E:u_unit}-\eqref{E:Div_T} 
converge to a limit $V$ in $H^r$. 
Since $r > 9/2$, we can pass to the limit in the equations
\eqref{E:u_unit}-\eqref{E:Div_T} satisfied by $V_k$ to conclude that $V$ solves \eqref{E:u_unit}-\eqref{E:Div_T} as well
(and that $\mathbf{\Psi}$ is in fact given in terms of $V$ by 
the same expressions that define $\mathbf{\Psi}_k$ in terms of $V_k$). By construction, 
$V$ takes the data $\mathcal{I}$.



\bibliography{References.bib}

\end{document}